\documentclass[12pt, English,titlepage]{article}

\usepackage{epsfig}
\usepackage{setspace}
\usepackage{ifthen}
\usepackage{graphicx}
\usepackage{amsmath}
\usepackage{amsfonts}
\usepackage{amssymb}
\usepackage{xypic}
\usepackage{lscape}
\usepackage{amsthm}
\usepackage{xy}
\usepackage{yhmath}
\usepackage{lscape}
\usepackage{amscd, bbold}
\usepackage{paralist}
\usepackage{graphics} 
\usepackage{epsfig} 

\usepackage[colorlinks=true]{hyperref}

\def\br#1{\left(#1\right)}

\newtheorem{theorem}{Theorem}

\newtheorem{corollary}[theorem]{Corollary}

\newtheorem{definition}[theorem]{Definition}
\newtheorem{ex}{Example}[section]

\newtheorem{lemma}[theorem]{Lemma}
\newtheorem{notation}[theorem]{Notation}

\newtheorem{proposition}[theorem]{Proposition}
\newtheorem{remark}[theorem]{Remark}

\def\br#1{\left(#1\right)}

\def\<#1,#2>{\langle #1,#2 \rangle}

 \def\bothID{\rlap{\hbox to.97\wd0{\hss\vrule height.06\ht0 width.82\wd0}}
 \copy0\rlap{\kern-.36\wd0\vrule height1.05\ht0 width.05\ht0}\kern.14\wd0}

\DeclareMathOperator{\diag}{diag} 

\DeclareMathOperator{\dist}{dist} 
\DeclareMathOperator{\diam}{diam} \DeclareMathOperator{\spec}{spec}
 \DeclareMathOperator{\ess}{ess}
\DeclareMathOperator{\vol}{vol} 
 
\DeclareMathOperator{\thick}{thick} 

\DeclareMathOperator{\dvol}{dvol} 
\DeclareMathOperator{\grad}{grad} \DeclareMathOperator{\exte}{ext}
\DeclareMathOperator{\inte}{int} \DeclareMathOperator{\Hom}{Hom}
 \DeclareMathOperator{\Spin}{Spin}
\DeclareMathOperator{\Cl}{Cl} \DeclareMathOperator{\SO}{SO}
\DeclareMathOperator{\tr}{tr} 
\DeclareMathOperator{\dil}{dil} \DeclareMathOperator{\APS}{APS}

\begin{document}

\title{Dirac Cohomology on Manifolds with Boundary and Spectral Lower Bounds}

\author{Simone Farinelli\\
        Aum\"ulistrasse 20\\
        CH-8906 Bonstetten\\
        Email: simone.farinelli@alumni.ethz.ch
        }

\maketitle

\begin{abstract}
Along the lines of the classic Hodge-De Rham theory a general
decomposition theorem for sections of a Dirac bundle over a compact
Riemannian manifold is proved by extending concepts as exterior
derivative and coderivative as well as as elliptic absolute and
relative boundary conditions for both Dirac and Dirac Laplacian
operators. Dirac sections are shown to be a direct sum of harmonic,
exact and coexact spinors satisfying alternatively absolute and
relative boundary conditions. Cheeger's estimation technique for
spectral lower bounds of the Laplacian on differential forms is
generalized to the Dirac Laplacian. A general method allowing to
estimate Dirac spectral lower bounds for the Dirac spectrum of a
compact Riemannian manifold in terms of the Dirac eigenvalues for a
cover of $0$-codimensional submanifolds is developed. Two applications
are provided for the Atiyah-Singer operator. First, we prove the existence
on compact connected spin manifolds of Riemannian metrics of unit volume
with arbitrarily large first non zero eigenvalue, which is an already known result. Second, we prove that
on a degenerating sequence of oriented, hyperbolic, three spin
manifolds for any choice of the spin structures the first positive non zero eigenvalue
is bounded from below by a positive uniform constant, which improves an already known result.
\end{abstract}
\noindent \text{Classification: }
 Primary: 58J50; Secondary: 53C27.\\
\text{Keywords: }
Dirac Operator, Spectral Lower Bounds, Dirac Cohomology.

\tableofcontents 

\section{Introduction}
When dealing with direct and indirect spectral theory on Riemannian
manifolds, the following question naturally arises. Given a formally
selfadjoint operator of Laplace type (or Laplacian for short) over a
compact Riemannian manifold, consider a finite open cover or a
decomposition into $0$-codimensional submanifolds with boundary and
add an appropriate elliptic boundary condition. Is there a general
principle allowing to find lower bounds of the spectrum of the
manifold in terms of the spectra of the pieces?\par To our knowledge the only answer
to this question known so far is a dissection principle, known also as domain monotonicity, which was originally
formulated for the Laplacian on functions on domains in
$\mathbf{R}^m$ by Courant and Hilbert (\cite{CoHi93} and
\cite{Cha84}). The remarkable fact is that it still holds for any
formally selfadjoint operator of Laplace type under Neumann boundary
conditions, as recognized for the first time by B\"ar \cite{Ba91} for the Dirac
Laplacian.\par

The main contribution of this paper is a new technique allowing to
estimate the lower spectral part of a general Dirac operator in
terms of the spectra of a finite cover under the appropriate
boundary conditions. The original idea in the case of differential
forms is due to Cheeger but unpublished, based on the Mayer-Vietoris
scheme, was  carried out in \cite{Go93}. In order to extend it to
the set up of Dirac bundles, a Dirac complex as in non commutative
differential geometry is introduced, as well as appropriate elliptic
local boundary conditions for both Dirac and
 Dirac Laplacian. Concepts like derivation and coderivation and
boundary conditions like  absolute and relative ones can be extended
from the context of differential forms  to that of Dirac sections.
\par If $(V,\left<\cdot,\cdot\right>,\nabla,\gamma)$ is a Dirac bundle over
the Riemannian manifold $(M,g)$, where $M$ is compact with boundary,
and if there exists a bundle isomorphism $T$ on $V$ anticommuting
 with $\gamma$ and with
the Dirac operator $Q$, for which $T^2=\mathbb{1}$, then, the tuple
$(V,\left<\cdot,\cdot\right>,\nabla,\gamma, \overline{\gamma})$ where $\overline{\gamma}:=iT\gamma$ defines a
$(1,1)$-Dirac bundle structure with corresponding Dirac operators
$Q$ and $\overline{Q}$. The operators
\begin{equation}
d:=\frac{1}{2}(Q-i\overline{Q})=
\frac{\mathbb{1}+T}{2}Q\quad\text{and}\quad
\delta:=\frac{1}{2}(Q+i\overline{Q})=\frac{\mathbb{1}-T}{2}Q
\end{equation}
can be seen derivative and coderivative on M, while the zero-order
boundary operators
\begin{equation}
B_{\pm}:=\frac{\mathbb{1}\mp T\gamma(\nu)}{2}
\end{equation}
play the role of the absolute ($B_{-}$) and relative ($B_{+}$)
boundary conditions on $\partial M$ for the Dirac operator $Q$. The
Dirac Laplacian can be decomposed as
\begin{equation}Q^2=d\delta+\delta d\end{equation} and the corresponding first
order boundary operators read as:
\begin{equation}B_{-}\oplus B_{-}d\quad \text{(absolute) }\quad
B_{+}\oplus B_{+}\delta \quad \text{(relative).} \end{equation}

\begin{theorem}[Orthogonal Decomposition of Dirac Sections]\label{Hodge0}
Let $(V,\left<\cdot,\cdot\right>,\nabla,\gamma)$ be a Dirac bundle over the compact
Riemannian manifold with boundary $(M,g)$ admitting a bundle isomorphism $T$  anticommuting with $\gamma$ and with
the Dirac operator $Q$, such that $T^2=\mathbb{1}$ holds. Let $C^{\infty}(M,V)$ denote the Dirac
spinors, i.e. the differentiable sections of $V$,
$\mathcal{H}_{B_{\pm}}(M,V)$ the harmonic,
$\Omega^d_{B_{\pm}}(M,V)$ the exact and
$\Omega^{\delta}_{B_{\mp}}(M,V)$ the coexact Dirac sections
satisfying the absolute ($B_-$) and the relative ($B_+$) boundary
conditions. Then, the following orthogonal decomposition holds:
\begin{equation}C^{\infty}(M,V)=\mathcal{H}_{B_{\pm}}(M,V)\oplus\Omega^d_{B_{\pm}}(M,V)\oplus\Omega^{\delta}_{B_{\mp}}(M,V).\end{equation}
\end{theorem}

\noindent This theorem generalizes Morrey's Theorem (cf. \cite{Mo56}
and \cite{Sc95}) for differential forms on manifold under the
relative or absolute boundary condition. By using this Hodge-De
Rham-like decomposition theorem a variational characterization of
the Dirac spectrum in terms of the Dirac spectrum on exact Dirac
sections can be derived. This is the technical result needed to
prove the following:

\begin{theorem}[Spectral Lower Bounds by Dissection]\label{Cheeger}
Let $(V,\left<\cdot,\cdot\right>,\nabla,\gamma)$ be a Dirac bundle over the
compact Riemannian manifold without boundary $(M,g)$ with Dirac operator $Q$. Assume the
existence  of a a bundle isomorphism $T$ on $V$ anticommuting with
$\gamma$ and with $Q$, such that $T^2=\mathbb{1}$ holds. Let $(U_j)_{j=0}^K$ be a collection of
closed sets whose interiors cover $M$. Choose and fix
$(\rho_j)_{j=1}^K$ a subordinate partition of unity and set
\begin{equation}
 \begin{split}
 U_{\alpha_0,\alpha_1,\dots,\alpha_k}&:=\bigcap_{i\in\{\alpha_0,\dots,\alpha_k\}}U_i\\
 N_1&:=\sum_{i,j=0}^K\dim\mathcal{H}_{B_{-}}(U_{i,j},V)\\
 N_2&:=\sum_{i,j,k=0}^K\dim\mathcal{H}_{B_{-}}(U_{i,j,k},V)\\
 N&:=N_1+N_2+1\\
 m_i&:=|\{j\neq i|\,U_j\cap U_i\neq\emptyset\}|\\
 C_{\rho}&:=\frac{1}{2}\max_{0\le i \le K}\sup_{x\in U_i} |\nabla
 \rho_i(x)|^2
 \end{split}
\end{equation}
\noindent For any closed set $U\subset M$ let $\lambda(U)$ denote
the smallest positive eigenvalue of the Dirac operator on exact
Dirac sections satisfying the absolute boundary condition $B_{-}$ on
$\partial U$. Then, the $N$-th positive eigenvalue of the Dirac
operator over $M$ has the following positive lower bound:
\begin{equation}
  \lambda_N(Q)\ge
  \frac{1}{\sqrt{\sum_{i=0}^K\br{
                        \frac{1}{\lambda^2(U_i)}
                        +4\sum_{j=0}^{m_i}\br{
                                           \frac{C_{\rho}}{\lambda^2(U_{ij})}
                                           +1}
                                           \br{\frac{1}{\lambda^2(U_i)}
                                           +\frac{1}{\lambda^2(U_j)}}
                        }
           }}.
\end{equation}
\end{theorem}
\noindent This is the generalization Cheeger's technique for the
Laplacian on differential forms (cf. \cite{Go93}). The lower
spectral bound method found can be applied to prove the new results
introduced by the following two subsections.

\subsection{Large First Eigenvalues}
Let $(M, g)$ be a compact, connected $n$ dimensional Riemannian
manifold. and $\lambda_1(\Delta_p^g))$ the smallest positive
eigenvalue of the Laplacian on $p$ forms. Hersch (\cite{Her70})
proved, that for functions on the sphere $S^2$ we have
\begin{equation}
\lambda_1(\Delta_0^g)\text{Vol}(S^2,g)\le 8\pi
\end{equation}
\noindent for every Riemannian metric $g$. In connection with this
result, Berger (\cite{Ber73}) asked whether there exists a constant
$k(M)$ such that
\begin{equation}\label{ineqlam}
\lambda_1(\Delta_0^g)\text{Vol}(M,g)^{\frac{2}{m}}\le k(M)
\end{equation}
for any Riemannian metric $g$ on a manifold $M$ of dimension $m$. Yang and Yau \cite{YaYa80}
proved that the inequality above holds for a compact surface $S$ of
genus $\Gamma$ with $k(S)=8\pi(\Gamma+1)$. Later, Bleecker
(\cite{Ble83}), Urakawa (\cite{Ura79}) and others constructed
examples of manifolds of dimension $m>3$ for which the inequality
(\ref{ineqlam}) is false.  Xu (\cite{Xu92}), and Colbois and Dodziuk
(\cite{CoDo94}) showed that  inequality (\ref{ineqlam}) is false for
every Riemannian manifold of dimension $m > 3$.  Tanno (\cite{Ta83})
posed the analogous question for forms of degree $p$, if there exist
a constant $k(M)$ such that
\begin{equation}\label{ineqForms}
\lambda_1(\Delta_p^g)\text{Vol}(M,g)^{\frac{2}{m}}\le k(M)
\end{equation}
for any Riemannian metric g on $M$. Pagliara and Gentile
(\cite{GP95}) showed that inequality (\ref{ineqForms}) is false for
$m>4$ and $2<p<m-2$. We can adapt now their proof to show

\begin{theorem}\label{Berger}
Every compact connected spin manifold $M$ of dimension $m\ge2$ without boundary admits for a given spin
structure $s$ metrics $g$ of volume one with arbitrarily large first
non zero Atiyah-Singer operator eigenvalue $\lambda_1(D^{(M,g)}_s)$.
\end{theorem}
When the proof of this theorem was written, the author was unaware that \cite {AJ11} had proved this result in the context of conformally covariant elliptic operators.

\subsection{Lower Spectrum of Degenerating Hyperbolic Three Manifolds}
According to Thurston's cusp closing Theorem (cf. \cite{Th79}),
every complete, non compact, hyperbolic, three manifold $M$ of
finite volume is the limit in the sense of pointed Lipschitz of a
sequence of compact, hyperbolic, three manifolds $(M_j)_{j\ge
0}$.\par The Laplace-Beltrami operator on p-forms is selfadjoint and
non negative. Its spectrum is contained in $[0,\infty[$ and can be
seen as the disjoint union of pure point spectrum i.e. eigenvalues
and continuous spectrum i.e. approximate eigenvalues or,
alternatively, as the disjoint union of non essential spectrum i. e.
isolated eigenvalues of finite multiplicity and essential spectrum
i. e. cluster points of the spectrum and eigenvalues of infinite
multiplicity.\par On the basis of Thurston's Theorem, we expect the
eigenvalues of ${\Delta}_p$ on $M_j$ to accumulate at points  of the
spectrum of ${\Delta}_p$ on $M$.\par In three dimensions the spectra
of functions and coexact 1-forms fully determines the spectra of
forms in all degree. In the case of functions, the results of
Donnely (\cite{Do80}) implied $\ess\spec( {\Delta}_0)=[1,\infty[$
and a sharp estimate for the number of eigenvalues of $M_j$ in any
interval $[1,1+x^2]$ was given by Chavel and Dodziuk \cite{CD93}. In
the case of 1-forms, Mazzeo and Phillips (\cite{MP90}) proved
$\spec({\Delta}_1)=[0,\infty[$ and the accumulation rate near $0$
was estimated by McGowan (\cite{Go93}). Later on, these results were
extended by Dodziuk and McGowan (\cite{DG95}), who gave an
asymptotic formula for the number of 1-form eigenvalues in an
arbitrary interval $[0,x]$.

\begin{theorem}[\bf Dodziuk, Mc Gowan ]\label{thm41}
On a degenerating sequence of hyperbolic compact three manifolds without boundary
$(M_j,g_j)_{j\ge 0}$ the lower eigenvalues of the Laplace-Beltrami
operator acting on $1-$forms accumulate near zero as the inverse of
the square of the diameter. More precisely, there exists an integer
$N_0\in \mathbf{N}_0$ such that
\begin{equation}\lambda_{N_0}(\Delta_1^{(M_j,g_j)})=\frac{O(1)}{\diam^2(M_j,g_j)}\quad(j\rightarrow
\infty).\end{equation}

\end{theorem}
Recall that $c_j=O(1)$ $(j\rightarrow\infty)$ if and only if $(c_j)_{j\ge0}$ is
a bounded sequence and that for a degenerating sequence of
hyperbolic manifolds $\diam(M_j,g_j)\uparrow+\infty$
$(j\rightarrow\infty)$.\par
An explicit lower bound for the first eigenvalue with respect to the diameter
has been recently provided by Jammes (cf. \cite{Ja12}).

\begin{theorem}[\bf Jammes]\label{thm42}
For any real $V>0$, there exists a constant $c(V)>0$ such that,
if $M$ is a three dimensional hyperbolic compact without boundary manifold of volume smaller than $V$,
whose thin part has $k$ components, then
\begin{equation}
\begin{split}
\lambda_{1}(\Delta_1^{(M,g)})&\ge\frac{c(V)}{\diam^3(M,g)\exp\left(2k\diam^3(M,g\right))}\\
\lambda_{k+1}(\Delta_1^{(M,g)})&\ge\frac{c(V)}{\diam^2(M,g)}.
\end{split}
\end{equation}
\end{theorem}

\begin{theorem}[\bf Jammes]\label{thm43}
For every non compact three dimensional hyperbolic manifold $M$ of bounded volume,
there exits a constant $c>0$ and a degenerating sequence of hyperbolic three compact without boundary manifolds
$(M_j,g_j)_{j\ge 0}$ converging to $M$ such that for all $j\ge0$
\begin{equation}\lambda_{1}(\Delta_1^{(M_j,g_j)})\ge\frac{c}{\diam^2(M_j,g_j)}.\end{equation}

\end{theorem}

In two dimensions the spectrum of $\Delta_0$ fully determines  the
spectra of forms of all degree. The analogous questions for surfaces
were studied by Wolpert (\cite{Wo87}), Hejahl (\cite{Hej90}) and Ji
(\cite{Ji93}) and a sharp estimate for the accumulation rate was
obtained by Ji and Zworski (\cite{JZ93}). In addition Colbois and
Courtois (\cite{CC89}, \cite{CC89bis}) proved that the eigenvalues
below the bottom of the essential spectrum are limits of eigenvalues
of $M_j$ for both Riemann surfaces and hyperbolic three
manifolds.\par Problems of this kind don't arise in dimensions
greater than or equal to four (cf. \cite{Gro79}), because the number
of complete hyperbolic manifolds of volume less than or equal to a
given constant is finite in this case.\par In the case  of the
classical Dirac operator B\"ar (cf. \cite{Ba00}) proved:
\begin{theorem}[\bf B\"ar]\label{thm51}
On a degenerating sequence of oriented, hyperbolic, three compact without boundary manifolds
$(M_j,g_j)_{j\ge0}$ for any spin structure $(s_j)_{j\ge0}$ on $M_j$
the lower eigenvalues of the Atiyah-Singer operator
$D_{s_j}^{(M_j,g_j)}$ do not accumulate. More precisely, there
exists an integer $N_0\in\mathbf{N}_0$ such that
\begin{equation}\left|\lambda_{N_0}(D_{s_j}^{(M_j,\,g_j)})\right|=O(1)
\quad(j\rightarrow\infty).\end{equation}
\end{theorem}
The different behaviour of spin Laplacian and Laplacian on forms is
due to topological reasons. We can improve Theorem \ref{thm51}
proving, by means of Theorem \ref{Cheeger}, that in Theorem
\ref{spinors3Manifold}  $N_0=1$ can be chosen and providing an
explicit lower bound for the first non zero eigenvalue of the Dirac
operator.

\begin{theorem} \label{spinors3Manifold} On a degenerating sequence of oriented, hyperbolic, three spin compact without boundary manifolds
for any choice of the spin structures  the lower eigenvalues of the
Atiyah-Singer operator do not accumulate and the first positive non
zero eigenvalue is bounded from below by a positive uniform constant $c>0$
\begin{equation}
\lambda_{1}(D_{s_j}^{(M_j,\,g_j)})\ge c.
\end{equation}
\end{theorem}

\section{Dirac Bundles}

The purpose of this chapter is to recall some basic definitions
concerning the theory of Dirac operators, establishing the
necessary self contained notation and introducing the standard
examples. The general references are \cite{LM89}, \cite{BW93},
\cite{BGV96} and \cite{Ba91}.
\subsection{Dirac Bundle}
\begin{definition}{\bf (Dirac Bundle)}
The quadruple $(V,\left<\cdot,\cdot\right>,\nabla,\gamma)$, where
  \begin{enumerate}
  \item[(i)] $V$ is a complex (real) vector bundle over the Riemannian manifold $(M,g)$ with
Hermitian (Riemannian) structure $\left<\cdot,\cdot\right>$,
  \item[(ii)] $\nabla:C^\infty(M,V)\to C^\infty(M,T^*M\otimes V)$ is a connection on $M$,
  \item[(iii)] $\gamma: \Cl(M,g)\to \Hom(V)$ is a {\it real} algebra bundle homomorphism from
the Clifford bundle over $M$ to the {\it real} bundle of complex
(real) endomorphisms of $V$, i.e. $V$ is a bundle of Clifford
modules,
 \end{enumerate}
    is said to be a {\it Dirac bundle}, if the following conditions are satisfied:
    \begin{enumerate}
    \item[(iv)] $\gamma(v)^*=-\gamma(v)$, $\forall v\in TM$ i.e. the Clifford multiplication
by tangent vectors is fiberwise skew-adjoint with respect to the
Hermitian (Riemannian) structure $\left<\cdot,\cdot\right>$.
    \item[(v)] $\nabla\left<\cdot,\cdot\right>=0$ i.e. the connection is Leibnizian (Riemannian). In
other words it satisfies the product rule:
      \begin{equation}
        d\left<\varphi,\psi\right>=\left<\nabla\varphi,\psi\right>+\left<\varphi,\nabla\psi\right>,\quad\forall
\varphi,\psi\in C^\infty(M,V).
      \end{equation}
    \item[(vi)] $\nabla\gamma=0$ i.e. the connection is a module derivation. In other words
it satisfies the product rule:
\begin{equation}
\begin{split}
\nabla(\gamma(w)\varphi)=\gamma(\nabla^gw)\varphi+\gamma(w)\nabla\varphi,&\quad\forall\varphi,\psi\in C^\infty(M,V),  \\
 &\quad\forall w\in C^\infty(M,\Cl(M,g)).
\end{split}
\end{equation}

\end{enumerate}
\end{definition}

Among the different geometric structures on Riemanniann Manifolds
satisfying the definition of a Dirac bundle (cf. \cite{Gi95}) the
canonical example is the spinor bundle.

\begin{definition}{\bf (Spin Manifold)}
 $(M,g,s)$ is called a {\it spin manifold} if and only if
 \begin{enumerate}
    \item $(M,g)$ is a $m$-dimensional oriented Riemannian manifold.
    \item $s$ is a {\it spin structure} for $M$, i.e. for $m\ge 3$ $s$ is a Spin$(m)$
          principal fibre bundle over $M$, admitting a double covering map
          $\pi:s\to SO(M)$ such that the following diagram commutes:
          \begin{equation}\xymatrix{s\times\Spin(m) \ar[d]^{\pi\times\Theta}\ar[r]& s\ar[d]^{\pi} \ar[r]& M \\
                      \SO(M)\times\SO(m) \ar[r] &\SO(M)\ar[ru] & \\
                     }\end{equation}
          where $\SO(M)$ denotes the $\SO(m)$ principal fiber bundle of the oriented basis of
          the tangential spaces, and $\Theta:\Spin(m)\to \SO(m)$ the
          canonical double covering. The maps $s\times \Spin(m)\rightarrow
          s$ and $\SO(M)\times\SO(m)\rightarrow \SO(M)$ describe the right
          action of the structure groups $\Spin(m)$ and $SO(m)$ on the
          principal fibre bundles $s$ and $SO(M)$ respectively.\\
          When $m=2$ a spin structure on $M$ is defined analogously with
          $\Spin(m)$ replaced by $\SO(2)$ and
          $\Theta:\SO(2)\rightarrow\SO(2)$ the connected two-sheet covering.
          When $m=1$ $\SO(M)\cong M$ and a spin structure is simply defined
          to be a two-fold covering of $M$.
 \end{enumerate}
The vector bundle over $M$ associated to $s$ w.r.t the { \it spin
representation} $\rho$ i.e.
      $$\Sigma M:=s\underset{\rho}{\times}{\mathbf{C}}^l\quad l:=2^{[\frac{m}{2}]}$$
      is called {\it spinor bundle} over $M$, see \cite{Ba91} page 18.
\end{definition}

\begin{ex}{\bf (Spinor bundle as a Dirac bundle)}\label{spinors}
Let $(M,g,s)$ be  a spin manifold of dimension $m$. We can make the
spinor bundle into a
Dirac bundle by the following choices:\\
$V:=\Sigma M$ : spinor bundle, rank$(V)=l$\\
$\left<\cdot,\cdot\right>$: Riemannian structure induced by the standard
Hermitian product in
$\mathbf{C}^l$ (which is $\Spin(m)$-invariant) and by the representation $\rho$.\\
$\nabla=\nabla^\Sigma$: spin connection = lift of the Levi-Civita
connection to the spinor bundle.
\begin{flalign*}
\gamma&:
      \begin{array}{lll}
        TM&\longrightarrow&\Hom(V)\\
        v&\longmapsto&\gamma(v)\text{, where}\quad\gamma(v)\varphi:=v\cdot\varphi
\quad(\cdot\text{ is the Clifford product})
      \end{array}
\end{flalign*}
We identified $TM$ with $SO(M)\underset{\alpha}{\times}\mathbf{R}^m$
($\alpha$ is the standard representation of $\mathbf{R}^m$) and
$\Sigma M$ with $s\underset{\rho}{\times}{\mathbf{C}}^l$. Since
$\gamma^2(v)=-g(v,v)\mathbb{1}$, by the universal property, the map
$\gamma$ extends uniquely to a real algebra bundle endomorphism
$\gamma:\Cl(M,g)\longrightarrow\Hom(V)$.
\end{ex}

\begin{ex}{\bf (Exterior algebra bundle as a Dirac Bundle)}\label{forms}
    Let $(M,g)$ be a $C^\infty$ Riemannian manifold of dimension $m$. The tangent and the
cotangent bundles are identified by the $\flat$-map defined by
$v^{\flat}(w):=g(v,w)$. Its inverse is denoted by $\sharp$. The
exterior algebra can be seen as a Dirac bundle after the following
choices:
    \begin{flalign*}
      V&:=\Lambda(T^*M)=\bigoplus_{j=0}^m \Lambda^j(T^*M):\,\text{exterior algebra over
$M$}\\
      \left<\cdot,\cdot\right>&:\,\text{Riemannian structure induced by $g$}\\
      \nabla&: \,\text{(lift of the) Levi Civita connection}\\
\gamma&:
      \begin{array}{lll}
        TM&\longrightarrow&\Hom(V)\\
        v&\longmapsto&\gamma(v):=\exte(v)-\inte(v)
      \end{array}
    \end{flalign*}
    where $\exte(v)\varphi:=v^{\flat}\wedge\varphi$
    and $\inte(v)\varphi:=\varphi(v,\cdot)$. Since $\gamma^2(v)=-g(v,v)\mathbb{1}$, by the universal property, the map $\gamma$ extends uniquely to a real algebra bundle endomorphism
    $\gamma:\Cl(M,g)\longrightarrow\Hom(V) $.
\end{ex}

\subsection{Dirac Operator and Dirac Laplacian}
\begin{definition}
Let $(V,\left<\cdot,\cdot\right>,\nabla,\gamma)$ be a Dirac bundle over the
Riemannian manifold $(M,g)$. The {\it Dirac operator}
$Q:C^\infty(M,V)\to C^\infty(M,V)$ is defined by
\begin{equation}\begin{CD}
   {C^{\infty}(M,V)} @>{\nabla}>> {C^{\infty}(M,T^*M\otimes V)} \\
    @V{Q:=\gamma\circ(\sharp\otimes\mathbb{1})\circ\nabla}VV  @VV{\sharp\otimes\mathbb{1}}V \\
   {C^{\infty}(M,V)} @<{\gamma}<< {C^{\infty}(M,TM\otimes V)}
 \end{CD}\end{equation}
The square of the Dirac operator $P:=Q^2:C^\infty(M,V)\to
C^\infty(M,V)$ is called the {\it Dirac Laplacian}.
\end{definition}

\begin{remark}
The Dirac operator $Q$ depends on the Riemannian metric $g$ and on
the homomorphism $\gamma$. If different metrics or homomorphisms are
considered, then the notation $Q={Q^g}_{\gamma}=Q^g=Q_{\gamma}$ is
utilized to avoid ambiguities.
\end{remark}

\begin{proposition}
The Dirac operator is a first order differential operator over $M$.
Its leading symbol is given by the Clifford multiplication:
\begin{equation}\sigma_L(Q)(x,\xi)=\imath\,\gamma(\xi^{\sharp})\end{equation}where $\imath:=\sqrt{-1}$.
The Dirac operator has the following local representation:
\begin{equation}
  Q(\varphi|_U)=\sum_{j=1}^m\gamma(e_j)\nabla_{e_j}\left(\varphi|_U\right)
\end{equation}
for a local orthonormal frame $\{e_1,\ldots,e_m\}$ for $TM|_U$ and a
section $\varphi\in
C^\infty(M,V)$.\\
The Dirac Laplacian is a second order partial differential operator
over $M$. Its leading symbol is given by the Riemannian metric:
 \begin{equation} \sigma_L(Q^2)(x,\xi)=g_x(\xi^\sharp,\xi^\sharp)\mathbb{1}_{V_x}\quad \forall x\in M,\, \xi\in
T^*_xM.\end{equation}
\end{proposition}

\begin{ex}[{\bf Atiyah-Singer Operator and Spin Laplacian}]
The Dirac operator in the case of spin manifolds $(M,g,S)$ is the
Atiyah-Singer operator $D_{\gamma}^g$ on the sections of the spinor
bundle $\Sigma M$. The Dirac Laplacian
$\Delta_{\gamma}^g:=(D_{\gamma}^g)^2$ is the spin Laplacian.
\end{ex}

\begin{ex}[{\bf Euler and Laplace-Beltrami Operators}]
The Dirac operator in the case of the exterior algebra bundle over
Riemannian manifolds $(M,g)$ is the Euler operator $d+\delta$ on
forms on $M$. The Dirac Laplacian
$\Delta:=(d+\delta)^2=d\delta+\delta d$ is the  Laplace-Beltrami
operator.
\end{ex}


\subsection{Dirac Complexes}

\begin{definition}[{\bf Normalized Orientation}]
Let $(V,\left<\cdot,\cdot\right>,\nabla,\gamma)$ be a Dirac bundle over the
{\it oriented} Riemannian manifold $(M,g)$. We consider a {\it
positively} oriented local orthonormal frame $\{e_1,\dots,e_m\}$
for $TM$. Then the product
\begin{equation}S:={\imath}^{[\frac{m+1}{2}]}\gamma(e_1)\cdots\gamma(e_m)\in\Hom(V)\end{equation}
is called the {\it normalized orientation} of the Dirac bundle.
\end{definition}

\begin{proposition} \label{sym}The normalized
orientation $S$ is well defined and independent of the choice of the
positively oriented local orthonormal frame. Moreover, it has the
following properties:
\begin{enumerate}
\item $S^2=\mathbb{1}$
\item $\nabla S=0$
\item $QS=(-1)^{m-1}SQ$.
\end{enumerate}
\end{proposition}
\begin{definition}{\bf (Dirac Complex)}
Let $Q$ be an operator of Dirac type for the vector bundle $V$
over the Riemannian manifold $(M,g)$ and $T\in \Hom(V)$. $(Q,T)$
is called a {\it complex of Dirac type} if and only if
\begin{enumerate}
\item $T^2=\mathbb{1}$
\item $QT=-TQ$.
\end{enumerate}
\end{definition}
\begin{notation}
\begin{equation}\Pi_{\pm}:=\frac{\mathbb{1}\mp T}{2} \qquad V_{\pm}:=\Pi_{\pm}(V)\qquad Q_{\pm}:=Q|_{C^{\infty}(M,V_{\pm})}.\end{equation}
\end{notation}
\begin{proposition}

\begin{enumerate}

\item $Q_{\pm}:C^\infty(M,V_{\pm})\longrightarrow C^\infty(M,V_{\mp})$
\item $Q=\left[\begin{matrix}0&Q_{-}\\Q_{+}&0\end{matrix}\right]:C^\infty(M,\underbrace{V_{+}\oplus V_{-}}_{V})\longrightarrow C^\infty(M,\underbrace{V_{+}\oplus V_{-}}_{V})$
\begin{equation*}\begin{CD}0@>>>{C^\infty(M,V_{+})}@>{Q_{+}}>>{C^\infty(M,V_{-})}@>{Q_{-}}>> {C^\infty(M,V_{+})} @>>> 0 \end{CD}\end{equation*}
 is a complex i.e. $Q_{-}Q_{+}=0$.
\item $(\Pi_{\pm}Q)^2=0\qquad\Pi_{+}Q+\Pi_{-}Q=Q\qquad\Pi_{+}Q\Pi_{-}Q+\Pi_{-}Q\Pi_{+}Q=Q^2.$

\end{enumerate}

\end{proposition}
\begin{ex}[{\bf Exterior Algebra in Even Dimensions}]\mbox{}\\
$T:=S$: normalized orientation.\\
$(d+\delta,S)$: {\it signature} complex.
\end{ex}
\begin{ex}[{\bf Exterior Algebra in any Dimensions}]\mbox{}\\
$T$ defined as $T|_{\Lambda^k(T^*M)}:=(-1)^k\mathbb{1}_{\Lambda^k(T^*M)}$\\
$(d+\delta,T)$: {\it (rolled up) De Rham} complex.
\end{ex}
\begin{ex}[{\bf Spinor Bundle in even Dimension}]\mbox{}\\
$T:=S$: normalized orientation.\\
$(D,S)$: spin complex.
\end{ex}
By Proposition \ref{sym} any Dirac bundle over an even dimensional
manifold can be made into a  complex of Dirac type by means of the
normalized orientation. In odd dimensions this is not possible,
because normalized orientation and Dirac operator commute.

The restriction of a Dirac bundle to a one codimensional
submanifold is again a Dirac bundle, as following theorem (cf.
\cite{Gi93} and \cite{Ba96}) shows.
\begin{theorem}\label{DiracBoundary}
Let $(V,\left<\cdot,\cdot\right>,\nabla,\gamma)$ be a Dirac bundle over the
Riemannian manifold $(M,g)$ and let $N\subset M$ be a one
codimensional submanifold with normal vector filed $\nu$. Then $(N,
g|_N)$ inherits a Dirac bundle structure by restriction. We mean by
this that the bundle $V|_N$, the connection $\nabla|_{C^{\infty}(N,
V|_N)}$, the real algebra bundle homomorphism
$\gamma_N:=-\gamma(\nu)\gamma|_{\Cl(N,g|_N)}$, and the Hermitian
(Riemannian) structure $\left<\cdot,\cdot\right>|_N$ satisfy the defining
properties (iv)-(vi). The quadruple $(V|_N, \left<\cdot,\cdot\right>|_N,
\nabla|_{C^{\infty}(N, V|_N)}, \gamma_N)$ is called the
\textbf{Dirac bundle structure induced on $N$} by the Dirac bundle
$(V,\left<\cdot,\cdot\right>,\nabla,\gamma)$ on $M$.
\end{theorem}

For a spin manifold of arbitrary dimension we will now construct a
vector bundle isomorphism $T$ which anticommutes with the
Atiyah-Singer operator $Q$ making a generic spin bundle to a complex
of Dirac type $(Q,T)$. Inspired by \cite{BGM05}, we embed a given manifold into a cylinder.

\begin{definition}[{\bf Generalized Cylinder}] Let $(M, g, \Spin(M))$ be a spin manifold
of dimension $n$, Riemannian metric $g$ and spin structure $\Spin(M)$. The
manifold $Z:=I\times M$, where $I$ denotes an interval of the real
line, equipped with the Riemannian metric $g^Z(u,x):=du^2\otimes g(x)$ and
with the spin structure $\Spin(Z):=\Spin(I)\times\Spin(M)$, with double covering map
\begin{equation}
\begin{split}
&\pi:\Spin(Z)=\Spin(I)\times\Spin(M)\rightarrow SO(Z)=SO(I)\times SO(M),\\ &\pi:=(\pi|_{\Spin(I)},\pi|_{\Spin(M)})
\end{split}
\end{equation}
is a spin manifold $(Z,g^Z,\Spin(Z))$ termed generalized cylinder, and
$i:SO(M)\rightarrow SO(Z),(e_1,\dots,e_n)\mapsto(\nu,e_1,\dots,e_n)$ denotes the canonical embedding.
\end{definition}
\noindent It can easily proved (cf. \cite{BGM05}, Chapter 5 and \cite{HMR15}, Chapter 2) that
\begin{proposition}\label{propT}
The original spin manifold and the generalized cylinder satisfy
following properties:
\begin{enumerate}
\item $\Spin(M)=\pi^{-1}(i(SO(M)))$.
 \item  $\gamma^{M}$ and $T:=i\gamma^Z\left(\frac{\partial}{\partial u}\right)$  anticommute. In fact, for all $v\in TM$
 \begin{equation}
     \gamma^{M}(v)T=-T\gamma^{M}(v).
 \end{equation}
\item $(\tilde{Q}^M,T)$ is a complex of Dirac type, where $\tilde{Q}^M:=Q^M$ if $n$ is even, and $\tilde{Q}^M:=\diag(Q^M,Q^M)$ if $n$ is odd, is termed the \textit{extrinsic Dirac operator}. In this context $Q^M$ is termed \textit{intrinsic Dirac operator}.
\item $\nabla^MT=0$.
\end{enumerate}
\end{proposition}

\section{Spectral Properties of the Dirac Operator}\label{Spectrum}
We consider Dirac bundles over compact manifolds, possibly with
boundary. The aim of this section is to summarize ``the state of the art'' concerning
the generic results about spectral results, especially in connection with boundary conditions.
 The existence of a regular discrete spectral resolution
for both Dirac and Dirac Laplacian operators under the appropriate
boundary conditions is a special case of the standard elliptic
boundary problems theory developed by Seeley (\cite{Sl66},
\cite{Sl69}) and Greiner (\cite{Gr70}, \cite{Gr71}). The general
references are \cite{Gr96} and  \cite{Ho85}. See \cite{BW93} and
\cite{Gi95} for the specific case of the Dirac and Dirac Lapacian
operators.

\subsection{Dirac and Dirac Laplacian Spectra on manifolds without boundary}
The Dirac operator $Q$ and the Dirac Laplacian $P$ for a Dirac
bundle $V$ over a compact Riemannian manifold  without boundary are
easily seen by Green's formula to be  symmetric operators for the
$C^{\infty}$-sections of Dirac bundle. Taking the completion of the
differentiable sections of $V$ in the Sobolev $H^1-$ and
respectively $H^2$-topology, leads to two selfadjoint operators in
$L^2(V)$.

\begin{theorem}\label{PQwithoutBoundary}
The Dirac $Q$ and the Dirac Laplacian $P$ operators of a Dirac
bundle over a compact Riemannian manifold $M$ without boundary have
a regular discrete spectral resolution with the same eigenspaces. It
exists a sequence  $(\varphi_j,\lambda_j)_{j \in \mathbb{Z}^* }$
such  that $(\varphi_j)_{j \in \mathbb{Z}^* }$ is an orthonormal
basis of $L^2(V)$ and that for every $j \in \mathbb{Z}^* $ it must
hold $Q\varphi_j=\lambda_j \varphi_j$ $P\varphi_j=\lambda_j
^2\varphi_j$ and $\varphi_j\in C^{\infty}(V)$. The eigenvalues of
the Dirac operator $(\lambda_j)_{j \in \mathbb{Z}^* }$ are  a
monotone increasing real sequence converging to $\pm\infty$ for
$j\rightarrow\pm\infty$. The eigenvalues of the Dirac Laplacian are
the squares of the eigenvalues of the Dirac operator and hence not
negative.
\end{theorem}

Therefore, for Dirac bundle over a manifold without boundary the
knowledge of the spectrum for the Dirac operator and the Dirac
Laplacian are equivalent. Moreover, in the case of a Dirac complex
the spectrum of the Dirac operator is symmetric with respect to the
origin.

\begin{proposition}\label{symm}
If there is an isomorphism $T$ for the Dirac bundle $V$
anticommuting with the Dirac operator $Q$, then the discrete
spectral resolution of Theorem \ref{PQwithoutBoundary} can be chosen
such that the equalities\,  $\lambda_{-j}=-\lambda_{+j}$\, and
\,$\varphi_{-j}=T\varphi_{+j}$\, hold for every $j\in\mathbb{N}^*$.
In particular, the dimension of the space of harmonic sections is
always even.
\end{proposition}

\begin{remark}An interesting consequence of Proposition \ref{symm}
and of Proposition \ref{propT} is that the spectrum of the \textit{extrinsic} classical
Dirac operator is symmetric with respect to the origin in any
dimension. The spectrum of the \textit{intrinsic} classical Dirac operator is always symmetric in even dimensions.  In odd dimensions nothing can be said a priori: there are cases, where the spectrum  of the intrinsic Dirac operator is not symmetric as Berger's spheres in dimension $\equiv 3$ mod $4$ (\cite{Ba96}), or some of the three dimensional compact Bieberbach manifolds beside the torus (see \cite{Pf00} for details), and cases where it is symmetric as Berger's spheres in dimension $\equiv 1$ mod $4$.
\end{remark}

\subsection{Dirac and Dirac Laplacian Spectra on Manifolds with Boundary}
The case of manifolds with boundary is more complex and the spectra
of the Dirac and Dirac Laplacians are no more equivalent as they are
in the boundaryless case. Moreover, while for the Dirac Laplacian it
is always possible to find local elliptic boundary conditions
allowing for a discrete spectral resolution, this is not always true
for the Dirac operator. The Dirac Laplacian $P$ for a Dirac bundle
$V$ over a compact Riemannian manifold with boundary is easily seen
by Green's formula to be a symmetric operator for the
$C^{\infty}$-sections of Dirac bundle if we impose the Dirichlet
boundary condition $B_D\varphi:=\varphi|_{\partial M}=0$ or the
Neumann boundary condition
$B_N\varphi=\nabla_{\nu}\varphi|_{\partial M}=0$. Taking the
completion of the differentiable sections of $V$ satisfying the
boundary conditions in the Sobolev $H^2$-topology, leads to a
selfadjoint operator in $L^2(V)$.
\begin{theorem}\label{PND}
The Dirac Laplacian $P$ of a Dirac bundle over a compact Riemannian
manifold $M$ with boundary under the Neumann or the Dirichlet
condition has a regular discrete spectral resolution
$(\varphi_j,\lambda_j)_{j \ge 0}$. This means that
$(\varphi_j)_{j\ge0}$ is an orthonormal basis of $L^2(V)$ and that
for every $j\ge 0$ it must hold $P\varphi_j=\lambda_j \varphi_j$,
$\varphi_j\in C^{\infty}(V)$, and $B\varphi_j=0$ for either $B=B_D$
or $B=B_N$. The eigenvalues $(\lambda_j)_{j\ge0}$ are a monotone
increasing real sequence bounded from below and converging to
infinity. The Dirichlet eigenvalues are all strictly positive. The
Neumann eigenvalues are all but for a finite number strictly
positive.
\end{theorem}

The situation for the Dirac operator is more subtle. Altough it is
-again by Green's formula- a symmetric operator under the Dirichlet
boundary condition, it is not selfadjoint. As a matter of fact the
Dirichlet boundary condition is elliptic for the Dirac Laplacian but
\textit{not} for the Dirac operator. If we are looking for
\textit{local} elliptic boundary conditions for the Dirac operator,
we need to introduce the following

\begin{definition} Let $(V,\left<\cdot,\cdot\right>,\nabla,\gamma)$ be a Dirac bundle over a manifold $M$
with boundary $\partial M$. The isomorphism $\chi\in
\Hom(V|_{\partial M})$ is called \textit {boundary chirality
operator} for the Dirac bundle if satisfies $\chi^2=\mathbb{1}$ and
anticommutes with the Clifford multiplication, i.e. $\chi \gamma(v)
+ \gamma(v)\chi=0$ for any $v\in TM|_{\partial M}$. The
corresponding boundary condition operator is given by
$B_{\pm}:=\frac{1}{2}(\mathbb{1}\mp\chi\gamma(\nu))$.
\end{definition}
In the \textit{even} dimensional case one can always find boundary
chirality operators for any Dirac bundle: it suffices to choose
$\chi:=S|_{\partial M}$, where $S$ denotes the normalized
orientation. For the special case of the exterior algebra bundle in
\textit{any} dimension the choice $\chi:=\exte(\nu)+\inte(\nu)$
leads to the \textit{absolute} and \textit{relative} boundary
conditions for differential forms which are ellipitic for the Euler
operator $d+\delta$.\par In the \textit{odd} dimensional case there are
obstructions to the existence of \textit{local}
boundary chirality operators for Dirac bundles. As a matter of fact,
if there exist a local elliptic boundary condition for the Dirac operator,
then $\tr(S)=0$. The non vanishing of the trace of the normalized orientation,
is therefore the topological obstruction, termed the \textit{Atiyah-Bott obstruction},
for the existence of local elliptic boundary conditions for the full Dirac Operator.
In even dimension this obstruction always vanishes because
the full Dirac operator and the normalized orientation always
anticommute. In odd dimensions the obstruction for the full Dirac operator can or cannot vanish.
It vanishes for the classic Dirac operator.
For the chiral Dirac operator, defined on the sections of the eigenbundles of the normalized
orientation the situation is complementary. In odd dimension the obstruction vanishes,
while in even ones it does not, see \cite{Gi84} page 248 and \cite{Gi95} page 102.
\par An elliptic boundary condition for both full and chiral Dirac operator always exists in
\text{any} dimension, but it is defined by mean of a zero order
pseudodifferential operator, the spectral projections of the Dirac
operator on the boundary. This is the famous Atiyah-Patodi-Singer
boundary condition (see \cite{Sl66}, \cite{BW93})). In a
neighbourhood of the boundary $\partial M$ it is possible to
decompose the Dirac operator as
\begin{equation}Q=\gamma(\nu)(\nabla_{\nu}+A).\end{equation}
Remark that it is \textit{not} necessary to assume that the
geometric structures are a product on this neighbourhood. The
operator $A|_{\partial M}$ is an operator of Dirac type for
$V|_{\partial M}$ over the boundaryless manifold $\partial M$. The
operator $A_{\APS}:=A|_{\partial M}+\frac{1}{2}H\mathbb{1}$, where $H$
denotes the mean curvature of the boundary, is an operator of Dirac
type for $\partial M$ and, by Theorem \ref{PQwithoutBoundary}, it
has a discrete regular spectral resolution $(\psi_j,\mu_j)_{j\ge
0}$. The subspace of $L^2(V)$ defined by
$E_{\mu}(A_{\APS}):=\ker(A_{\APS}-\mu\mathbb{1})$ is the eigenspace of
$A_{\APS}$ if $\mu$ is in the spectrum of $A_{\APS}$ and the zero
subspace otherwise.

\begin{definition}
Let $(V,\left<\cdot,\cdot\right>,\nabla,\gamma)$ be a Dirac bundle over the
{\it oriented} Riemannian manifold $(M,g)$ with Dirac operator $Q$
and normalized orientation $S$. The \textit{generalized
Atiyah-Patodi-Singer boundary condition} for $Q$ is given by
$B_{\APS}(\varphi|_{\partial M})=0$, where $B_{\APS}$ denotes the
orthogonal projection in $L^2(V|_{\partial M})$ onto
\begin{equation}\bigoplus_{\mu<0}E_{\mu}(A_{\APS})\oplus
\frac{1}{2}(\mathbb{1}-S)(E_0(A_{\APS})).\end{equation}
\end{definition}

The Dirac operator $Q$ for a Dirac bundle $V$ over a compact
Riemannian manifold  with boundary is easily seen by Green's formula
to be a symmetric operator for the $C^{\infty}$-sections of Dirac
bundle if we impose the boundary conditions $B_{\pm}$ induced by a
boundary chirality operator or by the generalized $\APS$ boundary
condition.  Taking the completion of the differentiable sections of
$V$ satisfying the boundary condition $B$ in the Sobolev
$H^1$-topology, leads to a selfadjoint operator in $L^2(V)$. Of
course, the associated first order boundary conditions for the Dirac
Laplacian are elliptic as well and lead to a self adjoint operator
with pure point spectrum if we define the domain of $P$ as the
completion of the differentiable sections of $V$ satisfying the
boundary conditions $B\oplus BQ$ in the Sobolev $H^2$-topology. In
\cite{FS98} it is given an elementary proof (with no reference to
the calculus of elliptic pseudodifferential operators as in
\cite{Ho85} or \cite{BdM71}) of the following result:

\begin{theorem}\label{QPWithBoundary}
The Dirac $Q$ and the Dirac Laplacian $P$ operators of a Dirac
bundle over a compact Riemannian manifold $M$ with boundary have
under the boundary conditions $B$ and $B\oplus BQ$ respectively, for
either $B=B_{\pm}$, (if a boundary chirality operator exists), or
$B=B_{\APS}$ a regular discrete spectral resolution with the same
eigenspaces. It exists a sequence $(\varphi_j,\lambda_j)_{j \in
\mathbb{Z}^* }$ such that $(\varphi_j)_{j \in \mathbb{Z}^* }$ is an
orthonormal basis of $L^2(V)$ and that for every $j \in \mathbb{Z}^*
$ it must hold $Q\varphi_j=\lambda_j \varphi_j$
$P\varphi_j=\lambda_j ^2\varphi_j$, $\varphi_j\in C^{\infty}(V)$,
$B(\varphi_j|_{\partial M})=0$ and $B((Q\varphi_j)|_{\partial
M})=0$. The eigenvalues of the Dirac operator $(\lambda_j)_{j \in
\mathbb{Z}^* }$ under the boundary condition $B$ are  a monotone
increasing real sequence converging to $\pm\infty$ for
$j\rightarrow\pm\infty$. The eigenvalues of the Dirac Laplacian are
the squares of the eigenvalues of the Dirac operator and hence not
negative.
\end{theorem}

Remark that the Dirac operator under the complementary
$\APS$-boundary condition, that, is the orthogonal projection from
$L^2(V)$ onto
\begin{equation}\bigoplus_{\mu>0}E_{\mu}(A_{\APS})\oplus
\frac{1}{2}(\mathbb{1}+S)(E_0(A_{\APS})),\end{equation} is only
symmetric but \textit{not selfadjoint} and thus must not have a
discrete real spectrum. The complementary $\APS$-boundary condition
is \textit{not elliptic}.
\par Extending the result in the boundaryless case, for a  Dirac
complex $(Q, T)$ preserving the boundary condition $B$, that is,
where $T$ anticommutes with $B$, the spectrum of the Dirac operator
is symmetric with respect to the origin.

\begin{proposition}\label{QTB}
If there is an isomorphism $T$ for the Dirac bundle $V$
anticommuting with the Dirac operator $Q$ and commuting with the
boundary condition $B$ for either $B=B_{+}$ or $B=B_{-}$ or
$B=B_{\APS}$, then the discrete spectral resolution of Theorem
\ref{QPWithBoundary} can be chosen such that the equalities\,
$\lambda_{-j}=-\lambda_{+j}$\, and \,$\varphi_{-j}=T\varphi_{+j}$\,
hold for every $j\in\mathbb{N}^*$. In particular, the dimension of
the space of harmonic sections satisfying the boundary condition $B$
is always even.
\end{proposition}

The local boundary conditions $B_{\pm}$ defined by mean of a
boundary chirality operator $\chi\in\Hom(V|_{\partial M})$ are
preserved by the Dirac complex $(Q, T)$ on $V$ if and only if
$\gamma(\nu)\chi$ and $T|_{\partial M}$ commute. This is always the
case for a Dirac bundle in \textit{even} dimensions, if we choose
$T:=S$ and $\chi:=S|_{\partial M}$, where $S$ is the normalized
orientation of the Dirac bundle. A special case, where Proposition
\ref{QTB} in \textit{any} dimension for local elliptic boundary
conditions, is the De Rham complex with either the absolute or
relative boundary conditions.\par The global boundary condition
$B_{\APS}$ defined by mean of the projection onto the eigenspaces of
the non positive eigenvalues of $A_{\APS}$ are preserved by the
Dirac complex $(Q, T)$ on $V$ if and only if $A_{\APS}$ and
$T|_{\partial M}$ commute. This is always the case for a Dirac
bundle in \textit{any} dimension, if we choose $T:=S$, where $S$ is
the normalized orientation of the Dirac bundle.

\section{Dirac Cohomology and Hodge Theory under Boundary Conditions}\label{Hodge-Dirac}
In this section we will prove Theorem \ref{Hodge0}. We will have to introduce for Dirac bundles concepts which mimick the situation for differential forms like derivation, coderivation, absolute and relative boundary conditions.
\begin{proposition}
Let $(V,\left<\cdot,\cdot\right>,\nabla,\gamma)$ be a Dirac bundle over
the Riemannian manifold $(M,g)$ with a bundle isomorphism $T$ on $V$ such that
$\overline{\gamma}:=iT\gamma$ anticommutes with $\gamma$ and with
the Dirac operator $Q$. The tuple $(V,\left<\cdot,\cdot\right>,\nabla,\gamma,
\overline{\gamma})$ defines a $(1,1)$-Dirac bundle structure with
corresponding Dirac operators $Q$ and $\overline{Q}$. The operators
\begin{equation}
d:=\frac{1}{2}(Q-i\overline{Q})=\frac{\mathbb{1}+T}{2}Q
\quad\text{and}\quad
\delta:=\frac{1}{2}(Q+i\overline{Q})=\frac{\mathbb{1}-T}{2}Q
\end{equation}
are called {\bf derivative and coderivative operators} on M and have following properties
\begin{enumerate}
\item The derivative defines a complex: $d^2=0$.
\item The coderivative defines a complex: $\delta^2=0$.
\item The Dirac operator can be decomposed as $Q=d+\delta$.
\item  The Dirac Laplacian can be decomposed as $P:=Q^2=d\delta+\delta d$.
\end{enumerate}

\noindent The zero-order boundary operators
\begin{equation}
B_{\pm}:=\frac{\mathbb{1}\mp T\gamma(\nu)}{2}
\end{equation}
define the {\bf absolute $B_{-}$ and relative $B_{+}$
boundary conditions} on $\partial M$ for the Dirac operator $Q$ and have following properties
\begin{enumerate}
\item $B_{+}\oplus B_{-}=\mathbb{1}$.
\item $B_{+}^2=B_{+}=B_{+}^*$.
\item $B_{-}^2=B_{-}=B_{-}^*$.
\item $\gamma(\nu)B_{\pm}=B_{\mp}\gamma(\nu)$ and $\gamma(\nu):\ker(B_{+})\oplus \ker(B_{-})\rightarrow \ker(B_{-})\oplus \ker(B_{+})$.
\end{enumerate}
The following Green's formula holds for all smooth sections $\varphi, \psi$ of the Dirac bundle
\begin{equation}\label{Green}
\begin{split}
(d\varphi,\psi)-(\varphi,\delta\psi)&=-\int_{\partial M}\,d\text{vol}_{\partial M}\left<\gamma(\nu)B_{-}\varphi,\psi\right>=\\
&=-\int_{\partial M}\,d\text{vol}_{\partial M}\left<\gamma(\nu)\varphi,B_{+}\psi\right>.
\end{split}
\end{equation}
\noindent For the Dirac Laplacian the corresponding first order boundary operators are $C_{-}:=B_{-}\oplus B_{-}d$ (absolute boundary condition) and $C_{+}:=B_{+}\oplus B_{+}\delta$ (relative boundary condition). In fact
\begin{enumerate}
\item The absolute boundary condition is preserved by the derivative operator: $B_{-}\varphi|_{\partial M}=0\Rightarrow B_{-}d\varphi|_{\partial M}=0$.
\item The relative boundary condition is preserved by the coderivative operator $B_{+}\varphi|_{\partial M}=0\Rightarrow B_{+}\delta\varphi|_{\partial M}=0$.

\end{enumerate}
\end{proposition}

\begin{proof} The properties of derivative and coderivative are a direct consequence of their definition where an isomorphism $T$
such that $(Q,T)$ is a Dirac complex was utilized. The properties
of the boundary conditions follows from the fact that
$(i\gamma(\nu)\overline{\gamma}(\nu))^2=\mathbb{1}$. The Green's
formula (\ref{Green}) follows from the corresponding Green's formulae for the
Dirac operators $Q$ and $\overline{Q}$. To prove the preservation of the absolute boundary condition by the derivative operator, we note that, by Green's formula
\begin{equation}\label{equa}
(d\varphi,\psi)=(\varphi,\delta\psi),
\end{equation}
for a $\varphi$ satisfying $B_{-}\varphi|_{\partial M}=0$ and \textit{any} $\psi$. Applying Green's formula to $d\varphi$ and $\psi$ we obtain
\begin{equation}\label{green2}
(dd\varphi,\psi)-(d\varphi,\delta\psi)=-\int_{\partial M}\,d\text{vol}_{\partial M}\left<\gamma(\nu)B_{-}d\varphi,\psi\right>
\end{equation}
The left hand side of (\ref{green2}) vanishes because of (\ref{equa}) and the fact that $d^2=0$. Thus, the boundary integral vanishes for all $\psi$ and so does $B_{-}d\varphi|_{\partial M}$. The proof of the preservation of the relative boundary condition under the coderivative operator reads analogously.
\end{proof}

After having introduced operators and boundary condition we would like to study the spectrum.
\begin{proposition}\label{spectralOp}
Let $H^1(M,V)$, $H^1_0(M,V)$ and $H^1_{B_{\pm}}(M,V)$ the domain of
definitions of  $d$, $d_0$, $d_{B_{\pm}}$ and $\delta$, $\delta_0$,
$\delta_{B_{\pm}}$ and $Q$, $Q_0$, $Q_{B_{\pm}}$, respectively. Let
$H^2(M,V)$, $H^2_0(M,V)$,  $H^2_{C_{\pm}}(M,V)$ the domain of
definitions of $P$, $P_0$ and $P_{B_{\pm}}$. They satisfy following
properties:
\begin{enumerate}
\item $d\subset\delta_0^*$, $d\subset\delta_0^*$,  $Q_0\subset Q_0^*$ and $P_0\subset P_0^*$.
\item $d_{B_{\pm}}^*=\delta_{B_{\mp}}$ and $\delta_{B_{\pm}}^*=d_{B_{\mp}}$.
\item $(Q,B_{\pm})$ are elliptic boundary value problems and $Q_{B_{\pm}}^*=Q_{B_{\pm}}$ are selfadjoint operators. If $M$ is compact, the operators $Q_{B_{\pm}}$ have discrete spectra and the corresponding eigensections are smooth.
\item $(P,C_{\pm})$ are elliptic boundary value problems and $P_{B_{\pm}}^*=P_{B_{\pm}}$ are selfadjoint operators. If $M$ is compact, the operators $P_{B_{\pm}}$ have non negative discrete spectra and the corresponding eigensections are smooth.
\end{enumerate}
\end{proposition}
\begin{proof} The proof is based on the Green's formula (\ref{Green}) and standard elliptic operator theory.
\end{proof}

\begin{theorem}[Orthogonal Decomposition of Dirac Sections]\label{Hodge}
Let $(V,\left<\cdot,\cdot\right>,\nabla,\gamma)$ be a Dirac bundle over the
compact Riemannian manifold $(M,g)$ admitting a bundle isomorphism $T$  anticommuting
 with $\gamma$ and with
the Dirac operator $Q$, such that $T^2=\mathbb{1}$ holds, and
\begin{itemize}
\item $\Omega(M,V):=C^\infty(M,V)$ be the smooth sections of the Dirac bundle on $M$,
\item $\mathcal{H}_{B_{\pm}}(M,V)$ be the harmonic sections of the Dirac bundle on $M$ satisfying the absolute or relative, respectively, boundary condition,
\item $\Omega_{B_{\pm}}^d(M,V):=\left\{\varphi\in\Omega_{B_{\pm}}(M,V)\,\big|\,\exists\psi\in\Omega(M,V)\,:\,d\psi=\varphi\right\}$ be the smooth exact Dirac sections on $M$ satisfying the absolute or relative, respectively, boundary condition,
\item $\Omega_{B_{\pm}}^{\delta}(M,V):=\left\{\varphi\in\Omega_{B_{\pm}}^p(M)\,\big|\,\exists\psi\in\Omega(M,V)\,:\,\delta\psi=\varphi\right\}$ be the smooth coexact Dirac sections on $M$ satisfying the absolute or relative, respectively, boundary condition.
\end{itemize}
Then, the following orthogonal decomposition holds:
\begin{equation}\label{HodgeDec}
C^{\infty}(M,V)=\mathcal{H}_{B_{\pm}}(M,V)\oplus\Omega^d_{B_{\pm}}(M,V)\oplus\Omega^{\delta}_{B_{\mp}}(M,V).
\end{equation}
\end{theorem}
\begin{proof} The proof is based on standard elliptic operator theory and the fact that derivative and coderivative operators preserve the absolute and the relative, respectively, boundary condition.
\end{proof}

\begin{definition}[{\bf Dirac Cohomology}] The group \begin{equation}
\mathbb{H}_{B_{\pm}}(M,V):=\{\omega\in\Omega_{B_{\pm}}(M,V)|d\omega=0\}/d\Omega_{B_{\pm}}^d(M,V)
\end{equation}
 is called
{\bf absolute}, respectively, {\bf relative Dirac cohomology} of the Dirac bundle.
\end{definition}
\noindent Since we will not need it going forward, we mention without proof the following result
\begin{theorem}\label{DeRham}
The mappings
\begin{equation}\label{iso}
I_{\pm}:\mathcal{H}_{B_{\pm}}(M,V)\rightarrow\mathbb{H}_{B_{\pm}}(M,V),\omega\mapsto I(\omega):=[\omega]
\end{equation}
are a natural isomorphisms between harmonic Dirac sections and Dirac cohomologies.
\end{theorem}
\begin{remark}
Of course decomposition (\ref{HodgeDec}) is a variation of the famous Hodge's Theorem and the isomorphims  (\ref{iso}) provide a similar result to De Rham's Theorem. The Dirac Cohomology is a Riemannian but not a topological invariant.
\end{remark}
To motivate the terminology introduced so far, we prove that in
the case of the Euler operator, for a particular choice of the
bundle isomorphism $T$ for the exterior algebra bundle, the
derivative and coderivative operators are the classical exterior
and interior differentiation for forms, the Dirac Cohomologies are
the De Rham cohomologies under the absolute and relative boundary
conditions and Theorem \ref{HodgeDec} the classical Hodge
decomposition theorem for differential forms on a manifold with
boundary.
\begin{proposition}
Let $(M,g)$ be an $m$ dimensional Riemannian manifold and $\{e_i\}_{i=1,\dots,m}$ be a local orthonormal field of $TM$.
Let $T_i:=\inte(e_i)\exte(e_i)-\exte(e_i)\inte(e_i)$, and $T:=\sum_{i=1}^mT_iP_i$, where the operator $P_i$ be the orthogonal projection onto\\ $W_i:=\{\exte(e_i)\varphi|\;\varphi\text{ is a local section of } \Lambda(T^*M)\}$. The operator $T$ can be extended to $M$ by a partition of unit argument and satisfies the following properties:
\begin{enumerate}
\item $T^2=\mathbb{1}$,
\item $\frac{\mathbb{1}-T}{2}(d+\delta)=d$
\item $\frac{\mathbb{1}+T}{2}(d+\delta)=\delta$,
\item $T(d+\delta)=-(d+\delta)T$,
\item Absolute boundary condition: $\inte(\nu)(\varphi)|_{\partial M}=0\Leftrightarrow B_{+}(\varphi)|_{\partial M}=0$,
\item Relative boundary condition  $\exte(\nu)(\varphi)|_{\partial M}=0\Leftrightarrow B_{-}(\varphi)|_{\partial M}=0$,
\end{enumerate}
where $B_{\pm}:=\frac{\mathbb{1}\mp T\gamma(\nu)}{2}$ for $\gamma(v):=\exte(v)-\inte(v)$.
\end{proposition}
\begin{proof} This can be verified by a direct computation.
\end{proof}

\section{Mayer-Vietoris's Scheme and Generalization of Cheeger's Spectral Estimate}
In this section we will prove Theorem \ref{Cheeger}. We first have to introduce several technicalities.
Let $M$ be a compact manifolds with boundary. If we impose the absolute boundary condition $B_{-}\phi|_{\partial M}=0$ on all Dirac eigensections  considered,  Theorem \ref{Hodge}  {\it and} the preservation of the first order absolute boundary condition under the derivative $d$ will allow for a special variational characterization of the spectra for Dirac and Dirac Laplacian. Inspired by results for Laplace-Beltrami operator on forms (cf. \cite{DG95}) and using Theorem \ref{Hodge}, one can prove

\begin{lemma}\label{lemma412}
Let $\lambda\in\spec(P_{C_{\pm}})$ be a non zero eigenvalue of the Dirac Laplacian under absolute or relative boundary conditions, and
\begin{itemize}
\item $E_{B_{\pm}}(\lambda):=\left\{\varphi\in\Omega_{B_{\pm}}(M,V)\,\big|\,P\varphi=\lambda\varphi\right\}$ be Dirac eigensections with eigenvalue $\lambda$,
\item $E^d_{B_{\pm}}(\lambda):=E_{B_{\pm}}(\lambda)\cap\Omega_{B_{\pm}}^d(M,V)$ be exact Dirac eigensections with eigenvalue $\lambda$,
\item $E^{\delta}_{B_{\pm}}(\lambda):=E_{B_{\pm}}(\lambda)\cap\Omega_{B_{\pm}}^{\delta}(M,V)$ be coexact Dirac eigensections with eigenvalue $\lambda$.
\end{itemize}
Then:
\begin{enumerate}
\item $E_{B_{\pm}}(\lambda)=E^d_{B_{\pm}}(\lambda)\oplus E^{\delta}_{B_{\pm}}(\lambda)$
\item $d:E^{\delta}_{B_{\pm}}(\lambda)\longrightarrow E^d_{B_{\pm}}(\lambda)$ and $\delta:E^d_{B_{\pm}}(\lambda)\longrightarrow E^{\delta}_{B_{\pm}}(\lambda)$  are isomorphisms between finite dimensional subspaces of $L^2(M,V)$.
\item $E^d_{B_{\pm}}(\lambda)=dE^{\delta}_{B_{\pm}}(\lambda)$ and $E^{\delta}_{B_{\pm}}(\lambda)=\delta E^d_{B_{\pm}}(\lambda)$.
\end{enumerate}
\end{lemma}
This lemma has an important consequence. The knowledge of the spectrum of the Dirac Laplacian on all exact (or coexact) Dirac sections implies the knowledge of the spectrum of the Dirac  Laplacian on all sections namely.
\begin{corollary}\label{corollary413}The spectrum of the Dirac Laplacian can be decomposed as
\begin{equation}
\spec(P_{B_{\pm}})=\{0\}\cup\spec(P_{B_{\pm}}\big|_{\Omega^d(M,V)})\cup\spec(P_{B_{\pm}}\big|_{\Omega^{\delta}(M,V)})
\end{equation}
The multiplicity of zero is the dimension of the absolute or relative, respectively, Dirac Cohomology. The multiplicity of an eigenvalue $\lambda>0$ is the sum of its multiplicities as exact and coexact eigenvalue.
\end{corollary}
Thus, to study the Dirac Laplacian and hence the Dirac spectrum, it suffices to study the spectrum of exact Dirac sections, whose  eigenvalues allow for the following minimax characterization.
\begin{proposition}\label{prop414}
If $(\lambda_i^d)_{i\ge0}:=\spec(P\big|_{\Omega_{B_{-}}^d(M,V)})$ are the eigenvalues of the Dirac Laplacian on exact sections, then
\begin{equation}
\lambda_i^d=\inf_L\sup_{\substack{\eta\in L\\\eta\neq0}}\left\{\frac{(\eta,\eta)}{(\varphi,\varphi)}\big|\,\varphi\in\Omega_{B_{-}}(M,V),\,d\varphi=\eta\right\}
\end{equation}
where $L$ varies over all $i$-dimensional subspaces of $\Omega_{B_{-}}(M,V)$.
\end{proposition}

\begin{proof} We take any $\varphi\in\Omega_{B_{-}}(M,V)$ such that $d\varphi=\eta$.
By Theorem \ref{Hodge}, any Dirac section $\varphi$ splits into the orthogonal sum $\varphi=h\oplus d\alpha\oplus\delta\beta$, where $h$ is an harmonic section, and $\alpha$, $\beta$ Dirac sections. Set $\psi:=\delta\beta\in\Omega_{B_{-}}^{\delta}(M,V)$. By the orthogonality of the decomposition
\[(\varphi,\varphi)\ge(\psi,\psi)\] and, by Green's formula and the coexactness of $\psi$:
\[(d\varphi,d\varphi)=(d\psi,d\psi)=(\delta d\psi,\psi)=(P\psi,\psi).\]
So,
\[\frac{(\eta,\eta)}{(\varphi,\varphi)}=\frac{(d\varphi,d\varphi)}{(\varphi,\varphi)}\le\frac{(P\psi,\psi)}{(\psi,\psi)}\]and
\[\inf_L\sup_{\substack{\eta\in L\\\eta\neq0}}\frac{(\eta,\eta)}{(\varphi,\varphi)}=\inf_R\sup_{\substack{\psi\in R\\\psi\neq0}}\frac{(P\psi,\psi)}{(\psi,\psi)}\]where $L$ varies over all $i$ dimensional subspaces of $\Omega_{B_{-}}(M,V)$ and $R$ over all $i$ dimensional subspaces of $\Omega_{B_{-}}^{\delta}(M,V)$. The right hand side  of this equation is the standard minimax characterization of $\lambda_i^{\delta}$, the $i$-th eigenvalue of coexact Dirac sections, which by Lemma \ref{lemma412} (ii) is equal to  the $i$-th eigenvalue $\lambda_{i}^d$ of exact sections.
\end{proof}

After having proved the variational characterization of Dirac Laplacian eigenvalues on exact sections satisfying the absolute boundary conditions, we possess now the technical tools to prove

\begin{proposition}\label{Cheeger2} Let $(V,\left<\cdot,\cdot\right>,\nabla,\gamma)$ be a Dirac bundle over a compact Riemannian manifold $(M,g)$.
We assume the existence of an isomorphism $T$  anticommuting with
with $\gamma$ and with the Dirac operator $Q$. Let $\mu(U)$ be the
smallest postive eigenvalue of the Dirac Laplacian $P$ on exact
Dirac sections satisfying the absolute boundary condition on $U$.
Moreover, for an an open cover of $M$ denoted by
$\{U_i\}_{i=0,\dots,K}$  we introduce the following notation:
\begin{itemize}
\item $U_{\alpha_0,\alpha_1,\dots,\alpha_k}:=\bigcap_{i\in\{\alpha_0,\dots,\alpha_k\}}U_i$.
\item $m_i:=|\{j\neq i|\,U_j\cap U_i\neq\emptyset\}|$.
\item $\{\rho_i\}_{i=0,\dots,K}$: a partition of unity subordinate to the open cover.
\item $C_{\rho}:=\frac{1}{2}\max_{i\in\{0,1,\dots,K\}}\sup_{x\in U_i}|\nabla\rho_i(x)|^2$.
\item $N_1:=\sum_{i,j=0}^K\dim\mathcal{H}_{B_{-}}(U_{i,j},V)$.
\item $N_2:=\sum_{i,j,k=0}^K\dim\mathcal{H}_{B_{-}}(U_{i,j,k},V)$.
\item $N:=N_1+N_2+1$.
\end{itemize}
The $N$-th eigenvalue of the Dirac Laplacian satisfies the following lower inequality
\begin{equation}
\mu_N(M)\ge\frac{1}{\sum_{i=0}^K\left(\frac{1}{\mu(U_i)}+4\sum_{j=0}^{m_i}\left(\frac{C_{\rho}}{\mu(U_{i,j})}+1\right)\left(\frac{1}{\mu(U_i)}+\frac{1}{\mu(U_j)}\right)\right)}
\end{equation}
\end{proposition}
\begin{proof}
Let $\{\Phi_i\}_{i\ge0}$ be an orthonormal basis of exact Dirac section in $C^{\infty}(M,V)$, where, for all $i\ge0$ $\Phi_i=d\chi_i$ and $\chi_i$ is coexact and thus unique. Therefore:
\begin{equation}
\begin{split}
\frac{(\Phi_N,\Phi_N)}{(\chi_N,\chi_N)}&=\frac{(d\chi_N,d\chi_N)}{(\chi_N,\chi_N)}=\frac{(\delta
d\chi_N,\chi_N)}{(\chi_N,\chi_N)}=\\& =\frac{((d\delta+\delta
d)\chi_N,\chi_N)}{(\chi_N,\chi_N)}=\frac{(P\chi_N,\chi_N)}{(\chi_N,\chi_N)}=\mu_N.
\end{split}
\end{equation}
\noindent Then, for every $\Phi\in\text{Span}(\{\Phi_i\}_{i=0,\dots,N})$, i.e. $\Phi=\sum_{i=0}^Na_i\Phi_i$, there exists a unique $\chi\in\text{Span}(\{\chi_i\}_{i=0,\dots,N})$, namely $\chi=\sum_{i=0}^Na_i\chi_i$, such that $d\chi=\Phi$. The uniqueness follows from the vanishing of a section which is at the same time exact and coexact.\\
Moreover,
\begin{equation}\label{in3}
\mu_N=\frac{(\Phi_N,\Phi_N)}{(\chi_N,\chi_N)}\ge\frac{(\Phi,\Phi)}{(\chi,\chi)}\ge\frac{(\Phi,\Phi)}{(\psi,\psi)},
\end{equation}
for all $\psi$ such that $d\psi=\Phi\in\text{Span}(\{\Phi_i\}_{i=0,\dots,N})$. As a matter of fact, by Theorem \ref{Hodge} $\psi=h\oplus d\alpha\oplus\delta\beta$. By denoting $\chi:=\delta \beta$, we have that $d\chi=\Phi $ and $(\psi,\psi)=(h,h)+(d\alpha,d\alpha)+(\chi,\chi)\le(\chi,\chi)$ and, hence, inequality (\ref{in3}). Therefore, a lower bound on $\frac{(\Phi,\Phi)}{(\psi,\psi)}$ for any par of $\Phi,\psi$ with $d\psi=\Phi\in\text{Span}(\{\Phi_i\}_{i=0,\dots,N})$ will give a lower bound on $\mu_N$.\par
We will construct a Dirac section $\overline{\psi}$ satisfying $d\overline{\psi}=\Phi$ in such a way that the $L^2$-norm of $\overline{\psi}$ is controlled in terms of the $L^2$-norm of $\Phi$. In order to do this we will be forced at two points during the proof to make specific choices for the coefficients $a_i$'s. Let us consider the following diagram
\begin{equation}\label{diag}
\xymatrix{
& \dots& \dots& \dots&\\
0 \ar[r]   &\Omega(M)\ar[r]^{r} \ar[u]^{d} &\Pi_{i}\Omega_{B_{-}}(U_i) \ar[r]^{s} \ar[u]^{d} &\Pi_{i,j}\Omega_{B_{-}}(U_{i,j})\ar[r]^{s} \ar[u]^{d} &\dots\\
0 \ar[r]  &\Omega(M)\ar[r]^{r}\ar[u]^d &\Pi_{i}\Omega_{B_{-}}(U_i) \ar[r]^{s}\ar[u]^d &\Pi_{i,j}\Omega_{B_{-}}(U_{i,j})\ar[r]^{s}\ar[u]^d &\dots\\
 & \dots \ar[u]^{d}& \dots \ar[u]^{d}&\dots\ar[u]^{d}&
 }
\end{equation}
Thereby, we utilize the notation:
\begin{itemize}
\item $r$: the restriction operator, which restricts global Dirac sections on $M$ to each open set of the cover according to
    \begin{equation}
     r(\omega):=\{\omega|_{U_i}\}_i.
    \end{equation}
\item $s$: the difference operator, which maps $\omega\in\Pi_{\alpha_0,\dots,\alpha_p}\Omega(U_{\alpha_0,\dots,\alpha_p})$ with components $\omega_{\alpha_0,\dots,\alpha_p}\in\Omega(U_{\alpha_0,\dots,\alpha_p})$ is defined as \begin{equation}
    (s\omega)_{\alpha_0,\dots,\alpha_p}:=\sum_{i=0}^{p+1}(-1)^i\omega_{\alpha_0,\dots,\hat{\alpha}_i,\dots,\alpha_p},
    \end{equation}
    where $\alpha_0,\dots,\hat{\alpha}_i,\dots,\alpha_p$ means that the index $\alpha_i$ has been dropped from the index sequence $\alpha_0,\dots,\alpha_p$.
\end{itemize}
The rows in diagram (\ref{diag}) are exact but the columns are not (in general). Since  we are interested in lower bounds for exact Dirac sections, we will pick $\Phi\in\text{Span}(\{\Phi_i\}_{i=0,\dots,N})$, the first $N$ exact eigensections. We restrict now $\Phi$ by means of $r$ to get $\{\Phi_i\}_i\in\Pi_i\Omega(U_i)$. Since $\Phi$ is exact, we can choose $\{\psi_i\}_i\in\Pi_i\Omega(U_i)$ so that $d\psi_i=\Phi_i$. Now, we can use the fact that we have a lower eigenvalue bound for exact sections on $U_i$ for all $i$ to choose $\psi_i$'s with bounded $L^2$ norm. We will then piece together these $\psi_i$'s into a section defined on all $M$. It is in general not true that $\psi_i=\psi_j$ on $U_{i,j}$, i.e. that $s\{\psi_i\}=0$. Therefore, we set $\{\omega_{i,j}\}=s\{\psi_i\}$, where $\omega_{i,j}:=\psi_i-\psi_j$ on $U_{i,j}$. Notice that
\begin{equation}
d\omega_{i,j}=d\psi_i-d\psi_j=\Phi-\Phi=0,
\end{equation}
so that, by Theorem \ref{Hodge}, we can write
\begin{equation}
\omega_{i,j}=h_{i,j}\oplus d\eta_{i,j},
\end{equation}
where $h_{i,j}$ is harmonic. We can choose appropriate coefficients $a_i$'s for $\Phi=\sum_{i=0}^Na_i\Phi_i\neq0$ so that $h_{i,j}=0$. The dimension of the space of such $\Phi$'s will be at least $N-N_1=N_2+1$. We pick the unique coexact $\eta_{i,j}$ such that $\omega_{i,j}=d\eta_{i,j}$. Therefore, by Proposition \ref{prop414}
\begin{equation}\label{instar}
\frac{(d\eta_{i,j},d\eta_{i,j})}{(\eta_{i,j},\eta_{i,j})}\ge\mu(U_{i,j})
\end{equation}
\noindent Next, let us consider $\{\nu_{i,j,k}\}=s\{\eta_{i,j}\}=\{(\eta_{j,k}-\eta_{i,k}+\eta_{i,j})|_{U_{i,j,k}}\}$ for which
\begin{equation}
\begin{split}
d\nu_{i,j,k}&=d\eta_{j,k}-d\eta_{i,k}+d\eta_{i,j}=\omega_{j,k}-\omega_{i,k}+\omega_{i,j}=\\
&=\psi_k-\psi_j-\psi_k+\psi_i+\psi_j-\psi_i=0,
\end{split}
\end{equation}
and therefore
\begin{equation}\label{diag2}
\xymatrix{
&\{\Phi_i\}\ar[r]^{s} & \{0\} &\\
&\{\psi_i\}\ar[u]^{d}\ar[r]^{s}&\{\omega_{i,j}\}\ar[u]^{d} \ar[r]^{s}&\{0\}\\
&\{\tau_i\}\ar[u]^{d}\ar[r]^{s}&\{\eta_{i,j,k}\} \ar[r]^{s}\ar[u]^{d}&\{\nu_{i,j,k}\}\ar[u]^{d}
 }
\end{equation}
We want to replace the $\psi_i$'s with some $\overline{\psi}_i$'s which are restrictions of a globally defined section on $M$ and such that on $U_i$
\begin{equation}
d\overline{\psi}_i=d\psi_i=\Phi_i.
\end{equation}
the exactness of the rows of diagram (\ref{diag}) would allow us, if all the $\nu_{i,j,k}$s were zero, to find $\{\tau_i\}\in\Pi_i\Omega(U_i)$ so that $s\{\tau_i\}=\{\eta_{i,j}\}=\{\tau_j-\tau_i|_{U_{i,j}}\}$. An explicit choice is given by
\begin{equation}
\tau_i:=\sum_{j=1}^K\rho_j\eta_{i,j},
\end{equation}
where $\{\rho_j\}_{j=0,\dots,K}$ is the partition of unity subordinate to the cover $\{U_j\}_{j=0,\dots,K}$. However, so far we can only claim that $d\nu_{i,j,k}=0$, i.e. that $\nu_{i,j,k}$ is closed. On the other hand $\nu_{i,j,k}$ is coexact, i.e. $\nu_{i,j,k}=\delta\alpha_{i,j,k}$. The mapping
\begin{equation}
\Phi\rightarrow\psi_i(\Phi)\rightarrow\omega_{i,k}(\Phi)\rightarrow\nu_{i,j,k}(\Phi)
\end{equation}
is linear in $\Phi$, which is in a space of dimension at least $N_2+1$. Therefore, we can choose $\Phi=\sum_{i=0}^Na_i\Phi_i\neq0$ such that
\begin{equation}\label{cond}
\nu_{i,j,k}(\Phi)=0\quad\text{ for all } i,j,k.
\end{equation}
As a matter of fact condition (\ref{cond}) represents $N_2$ linear equations in $N_2+1$ unknowns. So,
\begin{equation}
ds\{\tau_i\}=sd\{\tau_i\}=\{\omega_{i,j}\},
\end{equation}
and, if we take $\overline{\psi}_i:=\psi_i-d\tau_i$, then
\begin{equation}
s\{\overline{\psi}_i\}=\{\psi_j-\psi_i-d(\tau_j-\tau_i)\}=\{\psi_j-\psi_i-\omega_{i,j}\}=\{0\}
\end{equation}
Therefore, $\overline{\psi}_i=\overline{\psi}|_{U_i}$, where $\overline{\psi}$ is a globally defined section. Notice that $d\overline{\psi}_i =d\psi_i=\Phi_i$ on $U_i$. Since
\begin{equation}\label{i42}
(\overline{\psi},\overline{\psi})\le\sum_i(\overline{\psi}_i,\overline{\psi}_i),
\end{equation}
it follows
\begin{equation}\label{inev}
\frac{(\Phi,\Phi)}{\sum_i(\overline{\psi}_i,\overline{\psi}_i)}\le\frac{(\Phi,\Phi)}{(\overline{\psi},\overline{\psi})}.
\end{equation}
A lower bound on the left hand side of inequality (\ref{inev}) will give a lower eigenvalue bound for exact sections on $M$. Note that all norms are $L^2$-norms unless otherwise indicated, and are computed on the appropriate open sets. \par Being $\Phi_i$ the restriction of $\Phi$ to $U_i$, the variational characterization of the eigenvalues in Proposition \ref{prop414} implies
\begin{equation}
\frac{\|\Phi\|^2}{\|\psi_i\|^2}\ge\frac{\|\Phi_i\|^2}{\|\psi_i\|^2}\ge\mu(U_i),
\end{equation}
so that
\begin{equation}\label{i45}
\|\psi_i\|^2\le\frac{\|\Phi\|^2}{\mu(U_i)}.
\end{equation}
Both operators $Q$ and $\tilde{Q}$ satisfy the product rule for all smooth functions $f$ and sections $\varphi$
\begin{equation}
Q(f\varphi)=\gamma(\text{grad}f)\varphi+fQ\varphi\quad\text{and}\quad
\tilde{Q}(f\varphi)=\tilde{\gamma}(\text{grad}f)\varphi+f\tilde{Q}\varphi,
\end{equation}
so that the operator $d:=\frac{1}{2}(Q-i\tilde{Q})$
\begin{equation}
d(f\varphi)=\frac{\gamma-i\tilde{\gamma}}{2}(\text{grad}f)\varphi+fd\varphi.
\end{equation}
This formula allows to estimate $\|d\tau_i\|$:

\begin{equation}\label{i48}
\begin{split}
\|d\tau_i\|^2 &=\|d(\sum_j\rho_j\eta_{i,j})\|^2=\left\|\sum_j\frac{\gamma-i\tilde{\gamma}}{2}(\text{grad}\eta_{i,j}+\rho_jd\eta_{i,j}\right\|^2\le\\
&\le2\sum_j\left(\left\|\frac{\gamma-i\tilde{\gamma}}{2}(\text{grad}\eta_{i,j}\right\|^2+\left\|\rho_jd\eta_{i,j}\right\|^2\right)\le\\
&\le2\sum_j(C_j\|\eta_{i,j}\|^2+\|d\eta_{i,j}\|^2).
\end{split}
\end{equation}

\noindent Since $\eta_{i,j}$ fullfilling the condition (\ref{instar}) was chosen, we have
\begin{equation}\label{i49}
\|\eta_{i,j}\|^2\le\frac{\|d \eta_{i,j}\|^2}{\mu(U_{i,j})}=\frac{\|\psi_i-\psi_j\|^2}{\mu(U_{i,j})}\le\frac{2(\|\psi_i\|^2-\|\psi_j\|^2)}{\mu(U_{i,j})}.
\end{equation}
Assembling the inequalities (\ref{i49}), (\ref{i48}) and (\ref{i45}) into the definition of $\overline{\psi}_i$, we obtain
\begin{equation}
\begin{split}
\|\overline{\psi}_i\|^2&\le\|\psi_i\|^2+\|d\tau_i\|^2\le\\
&\le\|\psi_i\|^2+\|d(\sum_j\rho_j\eta_{i,j})\|^2\le\\
&\le\frac{\|\Phi\|^2}{\mu(U_i)}+4\sum_j\left(C_{\rho}\frac{\|\psi_i\|^2+\|\psi_j\|^2}{\mu(U_{i,j})}+\|\psi_i\|^2+\|\psi_j\|^2\right)\le\\
&\le\frac{\|\Phi\|^2}{\mu(U_i)}+4\sum_j\left(\frac{C_{\rho}\left(\frac{\|\Phi\|^2}{\mu(U_i)}+\frac{\|\Phi\|^2}{\mu(U_j)}\right)}{\mu(U_{i,j})}+\frac{\|\Phi\|^2}{\mu(U_i)}+\frac{\|\Phi\|^2}{\mu(U_j)}\right),
\end{split}
\end{equation}
\noindent and therefore
\begin{equation}
\frac{\|\overline{\psi}_i\|^2}{\|\Phi\|^2}\le\frac{1}{\mu(U_i)}4\sum_j\left(\frac{C_{\rho}\left(\frac{1}{\mu(U_i)}+\frac{1}{\mu(U_j)}\right)}{\mu(U_{i,j})}+\frac{1}{\mu(U_i)}+\frac{1}{\mu(U_j)}\right).
\end{equation}
Because of inequality (\ref{i42}) we finally obtain
\begin{equation}
\frac{\|\Phi\|^2}{\|\psi_\|^2}\le\frac{1}{\sum_{i=0}^K\left[\frac{1}{\mu(U_i)}4\sum_{j=0}^{m_i}\left(\frac{C_{\rho}\left(\frac{1}{\mu(U_i)}+\frac{1}{\mu(U_j)}\right)}{\mu(U_{i,j})}+\frac{1}{\mu(U_i)}+\frac{1}{\mu(U_j)}\right)\right]},
\end{equation}
which completes the proof.\\
\end{proof}

Theorem \ref{Cheeger} is now a direct consequence of Proposition
\ref{Cheeger2}, Lemma \ref{lemma412} and Corollary
\ref{corollary413}.

\section{Large First Dirac Eigenvalue: Proof of the Result}

\begin{proof}[Proof of Theorem \ref{Berger}]
We apply Theorem \ref{Cheeger} to the extrinsic Dirac operator as in Proposition \ref{propT} noting that
the spectral bound holds true for the intrinsic Dirac operator as well, because the spectra of both extrinsic and intrinsic Dirac Laplacians, possibly under the absolute boundary condition, are the same.\par We take a topological sphere $S^m$ and choose a metric $g_0$ on it, such
that $S^n$ looks like a cigar, where the middle part has length $3$.
In particular this middle part is a product for the metric $g_0$ ,
i.e. a cylinder $I\times S^{m-1}$. We then remove the half-sphere
$H_2$ at one end of the cigar and form a connected sum with $M$. The
resulting manifold is diffeomorphic to M and has a submanifold $N$,
with smooth boundary, naturally identified with $S^{m}\setminus H_2$. Let
$g_1$ be an arbitrary metric on M whose restriction to $N$ is equal
to $g_0|_N$. The manifold $N$ contains an open cylinder of length
$3$. We subdivide this cylinder into $3$ cylinders $Z_1, Z_2, Z_3$
of length $1$. Let $g_t$ be a metric on $M$ such that $g_t|_{M \setminus Z_2}
= g_1|_{M\setminus Z_2}$ and such that $Z_2 = I\times S^{m-1}$ becomes a
cylinder of length $t$. This is accomplished by replacing the unit
interval by the interval $[0, t]$ and using the product metric on
$Z_2$. Now $\text{Vol}(M, g_t) = a + bt$, where $a$ and $b$ are
positive real constants. We take the following open cover of $M$:
\begin{enumerate}
\item $U_1=H_1\cup Z_1$,
\item $U_2 = M\setminus \overline{H_1\cup Z_1\cup Z_2}$,
\item $U_3 = Z_1\cup\overline{Z_2}\cup Z_3$,
\end{enumerate}
which has the property that $U_1\cap U_2=\emptyset$, $U_1\cap
U_3=Z_1$, $U_2\cap U_3=Z_3$ and $U_1\cap U_2\cap U_3=\emptyset$. Let
$\mu_1(M_t)$ be the first positive eigenvalue of the Dirac Laplacian
on exact sections on $M_t=(M, g_t)$ for the given spin structure. To
estimate $\mu_1(M)$ we apply Theorem \ref{Cheeger} to $M_t$ and the
cover $\{U_1,U_2,U_3\}$. The eigenvalues $\mu(U_1),
\mu(U_2),\mu(U_{1,3})$ and  $\mu(U_{2,3})$ are independent of $t$.
Let $\lambda_k(O)$ be the $k$-th eigenvalue of the Dirac Laplacian
on $O$ under the absolute boundary condition  By using the
K\"unneth's formula for $m\ge 2$, we get the following inequality for
$\mu(U_3)$:
\begin{equation}
\begin{split}
\mu(U_3)&\ge\lambda_1(U_3)=\lambda_1(I\times S^{m-1})\\
&\ge\min_{i,j}\{\lambda_i(I)+\lambda_j(S^{m-1})\}=:C,
\end{split}
\end{equation}
where $C$ is a constant independent of $t$. If $m=3$, then $S^2$ has
no harmonic spinors (cf. \cite{Ba91}, \cite{Ba92}). In other
dimensions if the Riemannian metric on $S^{m-1}$ allows for non
trivial harmonic spinors, a small perturbation of the metric reduces
the harmonic spinors to the zero section (cf. \cite{BG92}).
Therefore, it is always possible to find a Riemannian metric for
which $\lambda_1(S^{m-1})>0$. Therefore, the constant $C$ is
strictly positive. From Theorem \ref{Cheeger} we get that
\begin{equation}
\mu_1(M_t)\ge\epsilon>0
\end{equation}
for an $\epsilon$ independent of $t$. The volume of $M_t$ is given
by $\text{Vo1}(M,g_t)=a+bt$ with constants $a, b>0$. Set $h_t =
(a+bt)^{\frac{2}{m}}$. For $(M, h_t)$ we have that $\text{Vol}(M,
h_t) = 1$ and $\lambda_1^2(D^{(M,h_t)}_s)
=(a+bt)^{\frac{2}{m}}\lambda_1^2(D^{(M,g_t})_s)$. This implies that
\begin{equation}
\lambda_1^2(D^{(M,h_t})_s) > \epsilon(a + bt)^{\frac{2}{m}}.
\end{equation}
\noindent Therefore $\lambda_1^2(D^{(M,h_t)}_s)\rightarrow+\infty$ as $t\rightarrow+\infty$. The proof is completed.\\
\end{proof}

\section{Lower Dirac Eigenvalues on
Degenerating Hyperbolic Three Dimensional Manifolds}

\subsection{The Geometry of Three Hyperbolic Manifolds}\label{curva}
A very readable survey of the geometry of compact, hyperbolic, three
manifolds and their degenerations is contained in Gromov
\cite{Gro79}. A very thorough discussion of this topic can be found
in Thurston \cite{Th79} or in Benedetti and Petronio
\cite{BP91}.\par The Kazhdan-Margulis decomposition gives a simple
insight of the geometrical structure of hyperbolic three manifolds,
particularly where the injectivity radius is small. There exists a
universal (i.e. depending only on the dimension) positive constant
$\mu$, called the Kazhdan-Margulis constant, for which the following
construction can always be carried out. Any hyperbolic manifold $M$
of finite volume splits into two
parts:\begin{equation}M=M_{]0,\mu]}\cup
M_{]\mu,\infty[}.\end{equation} $M_{]0,\mu]}$ is called the {\bf
thin part} and contains all points of $M$, whose injectivity radius
is smaller than or equal to $\mu$. The thin part is found to be a
finite union of tubes and cusps. A {\bf tube} $T$ is a tubular
neighbourhood of a closed geodesic. A {\bf cusp } $C$ is the warped
product $[0,+\infty[\times F$, equipped with the metric
$du^2+e^{-2u}ds^2$, where $F$ is a $2$-dimensional torus and $ds^2$
a flat metric on $F$. The points of $M$, where the injectivity
radius is bigger than $\mu$, form the so-called {\bf thick part}
$M_{]\mu,\infty[}$. The thick  part is non empty and connected.\par

The following theorem, due to Thurston (cf. \cite{BP91} page 197),
states that any complete hyperbolic three manifold of finite volume,
observed from its thick part, looks on its bounded part like a {\it
compact} hyperbolic three manifold.
\begin{theorem}[\bf Thurston]\label{thm THU}
Let $M$ be a complete, hyperbolic, three manifold with $p$ cusps,
$p\ge 1$, and of finite volume $\vol(M)$. Then, there is a sequence
$(M_j)_{j\ge0}$ of compact, hyperbolic, three manifolds having $p$
simple closed geodesics, whose lengths go to zero as
$j\rightarrow\infty$, such that $(M_j,x_j)$ converges to $(M,x)$ in
the sense of pointed Lipschitz, for appropriate $x_j$ and $x$
belonging to the thick part of $M_j$ and $M$, respectively. In
particular, $\vol(M_j)\uparrow \vol(M)$,
$\diam({M_j}_{\thick})\rightarrow \diam(M_{\thick})$ and if $M$ is
non compact, then $\diam(M_j)\uparrow\infty$.
\end{theorem}

\begin{definition}[\bf Pointed Lipschitz Convergence]
The {\bf dilatation} of a map $f:M\rightarrow N$ between two metric
spaces $M$ and $N$ wis defined as
\begin{equation}\dil(f):=\sup_{\substack{x,y\in M \\ x\neq y}}\frac{d(f(x),f(y))}{d(x,y)}\in
[0,+\infty].\end{equation} The {\bf Lipschitz distance} between $M$
and $N$ is the defined as
\begin{equation}d_L(M,N):=\inf\{|\log\dil(f)|+|\log\dil(f^{-1})|\}\end{equation}where the infimum is taken over all
Lipschitz homeomorphisms $f:M\rightarrow N$. The sequence
$(M_j,x_j)_{j\ge 0}$ of metric spaces $M_j$ with distinct points
$x_j\in M_j$ is said to converge to $(M,x)$ {\it in the sense of
pointed Lipschitz}, if and only if the following
condition is satisfied:\\
for every $r>0$ there exists a sequence $(\varepsilon_j)_{j\ge 0}$
of positive real numbers $\varepsilon_j\rightarrow 0^+$
$(j\rightarrow+\infty)$, such that
\begin{equation}d_L\left(B^{M_j}(x_j,r+\varepsilon_j),B^M(x,r)\right)\longrightarrow 0\quad
(j\rightarrow+\infty)\end{equation}where $B^M(x,r)$ denotes the ball
of radius $r$ in $M$ centered at $x$.
\end{definition}
As a matter of fact, Thurston shows that the compact manifolds
$M_j$, obtained by closing the cusps of an hyperbolic, complete, non
compact manifold $M$ using Dehn's surgery, support for all but for a
finite number of exceptions an hyperbolic metric and approximate
$M$.
\begin{definition}
If the limit manifold $M$ is non compact, then the sequence
$(M_j)_{j\ge 0}$ described above is called a {\bf degenerating
family of hyperbolic three manifolds}.
\end{definition}
A brief review of Riemannian metrics on tubes and cusps is needed
for the following. We refer to \cite{BP91} for more details. To keep
the notation simple, the manifold $M$ in Thurston's Theorem is
assumed without loss of generality to have only one cusp. There is a
positive $R_j$ for which the component of the thin part
$({M_j})_{]0,\mu]}$ of $M_j$ containing the closed simple geodesic
${\gamma}_j$, whose length
 ${\varepsilon}_j\rightarrow 0$ as $j\rightarrow\infty$, is the solid
torus
\begin{equation}T_j:=\left\{x\in M_j\,|\,\dist(x,{\gamma}_j)\le R_j\right\}.\end{equation}
This torus is the quotient of a solid hyperbolic cylinder
$\tilde{T_j}$ in the universal cover $\mathbf{H}^3$ of $M_j$ by the
action of an infinite cyclic group of isometries generated by an
hyperbolic twist of length ${\varepsilon}_j$ and angle ${\rho}_j\in
[o,\pi[$. Some non trivial facts about hyperbolic geometry accounted
for example in Colbois and Courtois (\cite{CC89}) or in Dodziuk and
McGowan (\cite{DG95}) force the distinguished constants
$R_j,{\varepsilon}_j,{\rho}_j$ to satisfy the following
inequalities:
\begin{equation}\label{ineqTube}
\begin{split}
&D_1e^{-2R_j}\le{\varepsilon}_j\le D_2e^{-2R_j}\\
&E_1e^{-R_j}\le{\rho}_j\le E_2e^{-R_j}
\end{split}
\end{equation}
where $D_{j}$ and $E_{j}$ $(j=1,2)$ are positive constants. In terms
of Fermi coordinates $(r,t,\theta)$ with respect to the geodesic
$\tilde{{\gamma}_j}$, the lift of ${\gamma}_j$ in $\mathbf{H}^3$, we
can write the twist
as\begin{equation}A_{{\gamma}_j}:(r,t,\theta)\rightarrow(r,t+{\varepsilon}_j,\theta+{\rho}_j)\end{equation}
and the metric on $\tilde{T_j}$ as
\begin{equation}\tilde{g_j}=dr^2+{\cosh}^2 r\, dt^2+{\sinh}^2 r\, d{\theta}^2,\end{equation}
where $r\in]0,R_j],\,t\in[0,{\varepsilon}_j]$ and
$\theta\in[0,2\pi]$. If we change the radial coordinate by
$u:=R_j-r\in [0,R_j[$ and introduce the following auxiliary
functions
\begin{equation}
\begin{split}
{\varphi}_j(u)&:=\frac{1}{4}(e^{2u}-1){\cosh}^{-2}\,R_j\,(e^{-2R_j}(1+e^{2u})+2)\\
{\psi}_j(u)&:=\frac{1}{4}(e^{2u}-1){\sinh}^{-2}\,R_j\,(e^{-2R_j}(1+e^{2u})-2),
\end{split}
\end{equation}
\noindent the metric on $\tilde{T_j}$ becomes in the new coordinates
\begin{equation}\tilde{g_j}=du^2+e^{-2u}\left\{(1+{\varphi}_j(u)){\cosh}^2\,R_j\,dt^2+(1+{\psi}_j(u)){\sin
h}^2\,R_j\,d{\theta}^2\right\},\end{equation} from which the
similarity with the warped product metric
\begin{equation}\label{warped}\widetilde{g_j^\prime}=du^2+e^{-2u}\left\{{\cosh}^2\,R_j\,dt^2+{\sinh}^2\,R_j\,d{\theta}^2
\right\}\end{equation} \noindent is evident. As a matter of fact
${\varphi}_j=o(1)$ and ${\psi}_j=o(1)$ pointwise on $[0,R_j]$ and,
in view of Thurston's Theorem, $T_j$ is expected to become a cusp in
the limit $j\rightarrow+\infty$.  \par We conclude by some observations about
the fibers of the tubes and the cusp. The warped product metric on
$T_j$ writes as\[g_j^\prime=du^2+e^{-2u}ds_j^2\] where
$(F_j,ds_j^2)$ is a flat torus. More exactly :
$\tilde{F_j}=\mathbf{R}^2$ and $F_j=\tilde{F_j}/ \sim $  w.r.t. the
identifications in polar coordinates
$(t,\theta)\sim(t+\varepsilon_j,\theta+\rho_j)$ and
$(t,\theta)\sim(t,\theta+2\pi)$ for all $(t,\theta)$ and the metric
in the universal cover $\mathbf{R}^2$ is given by
\[\widetilde{ds_j^2}=\cosh^2 R_jdt^2+\sinh^2 R_jd\theta^2.\]
Colbois-Courtois \cite{CC89} proved:
\begin{proposition}\label{prop34} A subsequence of $(F_j,ds_j^2)_{j\ge 0}$ converges in the sense of
Lipschitz to the flat torus $(F,ds^2)$, where $C=[0,+\infty[\times
F$ is the cusp in the limit manifold $(M,g)$.
\end{proposition}

\subsection{Spectrum of the Tube}
We want to compute the eigenvalues of the Dirac Laplacian under the
absolute boundary condition for $U:=\left\{x\in M\,|\,r_0\le\dist(x,{\gamma})\le R_0\right\}$, a piece of tube $T$ of a compact hyperbolic spin three manifold $M$. To do so
we introduce a local o.n. frame for the spinor bundle over $U$. Recall from \ref{curva} that the points of the tube $T$ at geodesic distance $u$ from $\partial T$ form a flat torus $F_u$, whose metric in its universal cover $\Tilde{F_u}$ is $\Tilde{ds^2}=f^2(u)dt^2+h^2(u)d\theta^2$, where $f(u):=\cosh(R-u)$ and $h(u):=\sinh(R-u)$. The Dirac  bundle structure on the odd dimensional manifold $T$ induces on each $1$-codimensional submanifold $F_u$ a unique Dirac bundle structure (see Theorem \ref{DiracBoundary}).
\begin{proposition}\label{prop53} Let us denote with $\partial_u,\partial_t,\partial_\theta$ the partial derivatives w.r.t. $u,t,\theta$ on the local frame for $TM|_{U}$ corresponding to these coordinates.
It exist a local o.n. frame $\{s_1,\dots,s_l\}$ for the spinor bundle $\Sigma M|_U$, whose rank is $l$, with the following properties:
\begin{enumerate}
\item [(i)]$\nabla^{F_u}s_k=0\quad(1\le k\le l,\,u\in[r_0,R_0])$ i.e. the spinors are parallel in each fiber,
\item [(ii)] $\nabla^M_{\partial_u}s_k=0\quad(1\le k\le l,\,u\in[r_0,R_0])$ i.e. the spinors are parallel to the radial direction.
\item[(iii)] $B_-s_1=s_1,\dots,B_-s_{\frac{l}{2}}=s_{\frac{l}{2}}$ and $B_-s_{\frac{l}{2}+1}=0,\dots,B_-s_l=0$, where $B_{\pm}:=\frac{\mathbb{1}\mp T\gamma^M(\partial_u)}{2}$ and $T$ as in Proposition \ref{propT}.
\end{enumerate}
\end{proposition}

\begin{proof}
Being $F_u$ flat and the spin connection the lift of the Levi-Civita connection, the parallel transport on $F_u$ doesn't depend (locally!) on the path. We consider $x_0=(u_0,t_0,\theta_0)\in F_{u_0}$ and an o.n. basis $s_1(u_0,t_0,\theta_0),\dots,s_l(u_0,t_0,\theta_0)$ of $V_{x_0}$ where $V:=\Sigma M$. There exist a neighbourhood of $x_0$ in $F_{u_0}$, where, without being worried about paths, we can set for any $1\le k \le l$
\begin{equation}s_k(u_0,t,\theta):=\Pi_{F_{u_0}}^{(t_0,\theta_0)\rightarrow (t,\theta)}s_k(u_0,t_0,\theta_0),\end{equation}
where $\Pi_{F_u}$ denotes the parallel transport on $F_u$. Since the parallel transport is an isometry, the frame $\{s_k(u_0,\cdot,\cdot)\}_{1\le k\le l}$ is a local o.n. frame for $V|_{F_{u_0}}$ satisfying by definition the property (i) for $u=u_0$.\\
Now we set
\begin{equation} s_k(u,t,\theta):=\Pi_M^{u_0\rightarrow u}s_k(u_0,t,\theta),\end{equation}
where $\Pi_M^{u_0\rightarrow u}$ denotes the parallel transport on the tube along the $u$-lines. Since the parallel transport is an isometry, the frame $\{s_k\}_{1\le k\le l}$ is a local o.n. frame for $V|_U$, satisfying by definition property (ii). We choose $u_0:=r_0$ and $u$ can vary in $[r_0,R_0]$.\par
The fact that property (i) holds for any $u\in[r_0,R_0]$, that is
\begin{equation}\nabla^{F_u}\Pi_M^{u_0\rightarrow u}s_k(u_0,t,\theta)=0,\end{equation}
follows from the formulae
\begin{equation}\nabla^{F_u}_{\frac{1}{f(u)}\partial_t(u)}\Pi_M^{u_0\rightarrow u}=\frac{f(u_0)}{f(u)}\Pi_M^{u_0\rightarrow u}\nabla^{F_{u_0}}_{\frac{1}{f(u_0)}\partial_t(u_0)}\end{equation}
and
\begin{equation}\nabla^{F_u}_{\frac{1}{h(u)}\partial_\theta(u)}\Pi_M^{u_0\rightarrow u}=\frac{h(u_0)}{h(u)}\Pi_M^{u_0\rightarrow u}\nabla^{F_{u_0}}_{\frac{1}{h(u_0)}\partial_\theta(u_0)}.\end{equation}
Since $(T\gamma(\partial_u))^2=\mathbb{1}$, we can choose $s_1(u_0,t_0,\theta_0),\dots,s_l(u_0,t_0,\theta_0)$ satisfying (iii) in $x_0$. Since $\nabla^M$ commutes with $T$ (see Proposition \ref{propT}) and with $\gamma(\partial_u)$, property (iii) holds for all $s_1,\dots,s_l$ over $U$, which are obtained by parallel transport. Therefore, property (iii) holds true.
\end{proof}
\begin{remark}\label{remark61}
The domain of definition $O\subset U$ of such a local o.n. frame $\{s_1,\dots,s_l\}$ is typically the image of an open subset $U_{u_0}\subset F_{u_0}$ under the exponential flow in the piece of the tube normal to the fiber $F_{u_0}$.
\end{remark}

\begin{lemma}\label{lemma54} Let $D^F:=D^{F_u}$ and $\Delta_s^F:={(D^F)}^2$ denote the Dirac operator and, respectively, the spin Laplacian on $F_u$, and $H=-\frac{1}{2}\partial_u(\log fh)$ the mean curvature of $F_u$ in $M$.
For any spinor $\sigma$ over the piece of the tube $U$, the spin Laplacian writes as
\begin{equation}\Delta_s^M\sigma=\Delta_s^F\sigma+[D^F,\nabla^M_{\partial_ u}]\sigma
-(\nabla^M_{\partial_u})^2\sigma+(\partial_uH-H^2)\sigma+2H\nabla^M_{\partial_u}\sigma.\end{equation}
\end{lemma}
\begin{proof}
According to B\"ar \cite{Ba96} the Dirac operator  on the tube writes as
\begin{equation}D^M\sigma=\gamma(\partial_u)D^F\sigma-H\gamma(\partial_u)\sigma+\gamma(\partial_u)\nabla^M_{\partial_u}\sigma.\end{equation}
So, for the spin Laplacian we have
\begin{equation}\begin{split}
\Delta_s\sigma
&=D^MD^M\sigma=D^M(\gamma(\partial_u)D^F\sigma-H\gamma(\partial_u)\sigma+\gamma(\partial_u)\nabla^M_{\partial_u}\sigma)=\\
&=\gamma(\partial_u)D^F(\gamma(\partial_u)D^F\sigma-H\gamma(\partial_u)\sigma+\gamma(\partial_u)\nabla^M_{\partial_u}\sigma)+\\
&\quad-H\gamma(\partial_u)(\gamma(\partial_u)D^F\sigma-H\gamma(\partial_u)\sigma+\gamma(\partial_u)\nabla^M_{\partial_u}\sigma)+\\
&\quad+\gamma(\partial_u)\nabla^M_{\partial_u}(\gamma(\partial_u)D^F\sigma-H\gamma(\partial_u)\sigma+\gamma(\partial_u)\nabla^M_{\partial_u}\sigma)=\\
&=(D^F)^2\sigma+(D^F\nabla^M_{\partial_u}-\nabla^M_{\partial_u}D^F)\sigma-(\nabla^M_{\partial_u})^2\sigma+(\partial_u H-H^2)\sigma+\\
&\quad +2H\nabla^M_{\partial_u}\sigma,
\end{split}\end{equation}which is the assertion of the lemma. We used of course that\begin{equation}\nabla^M_{\partial_u}\partial_u=0\end{equation} and that \begin{equation}-\gamma(\partial_u)\nabla_{\partial_u}^M(H\gamma(\partial_u)\sigma)=\partial_uH\sigma+H\nabla_{\partial_u}^M\sigma.\end{equation}
\end{proof}

\begin{proposition}\label{prop55}Let $c_{ri}:=\left<\gamma^F(f^{-1}\partial_t)s_r,s_i\right>$ and $d_{ri}:=\left<\gamma^F(h^{-1}\partial_\theta)s_r,s_i\right>$ for $1\le r,i \le l$.
Let $\{s_1,\dots,s_l\}$ be the local o.n. frame for the spinor bundle over the piece of tube $U$ defined in Proposition \ref{prop53}. Under the decomposition $\sigma=\sum_{k=1}^l\sigma^ks_k$ , we obtain the following equivalences
\begin{enumerate}
\item[(i)] Eigenvalue equation:
\begin{equation}\begin{split}(\Delta_s-\lambda)\sigma=0&\Leftrightarrow(\Delta_0-\lambda)\sigma^k-\partial_u^2\sigma^k+(\partial_uH-H^2)\sigma^k+ \\
&\quad+2H\partial_u\sigma^k-\sum_{i\neq k}\left[(f\partial_u(f^2)+\partial_uf)c_{ik}\partial_t\sigma^i\right.+\\
&\quad\left.-(h\partial_u(h^2)+\partial_uh)d_{ik}\partial_\theta\sigma^i\right]=0\\
&\quad(1\le k\le l). \end{split}
\end{equation}
\item[(ii)] Absolute boundary condition:
\begin{equation}
\begin{split}(B_-\sigma)|_{\partial U}=0,&\quad (B_-D^M\sigma)|_{\partial U}=0\\
&\Leftrightarrow\\
\sigma^k=0,\quad(\partial_u-H)\sigma^{k+\frac{l}{2}}=0&\quad(u=r_0,R_0 \text{ and } 1\le k \le\frac{l}{2}).
\end{split}
\end{equation}
\end{enumerate}
\end{proposition}
\begin{proof}
We insert the decomposition $\sigma=\sum_{k=1}^l\sigma^ks_k$, where $\sigma^k=\sigma^k(u,t,\theta)$, in the equation $\Delta_s\sigma=\lambda\sigma$.
 We represent $\Delta_s$ using Lemma \ref{lemma54}:
\begin{equation}\Delta_s^M\sigma=\Delta_s^F\sigma+[D^F,\nabla^M_{\partial_u}]\sigma-(\nabla^M_{\partial_u})^2\sigma+(\partial_u H-H^2)\sigma+2H\nabla^M_{\partial_u}\sigma.\end{equation}
Using the properties of the local o.n. frame $s_1,\dots,s_l$ described in the Proposition \ref{prop53}, we find:
\begin{equation}
\begin{split}
&D^F\sigma=\sum_{k=1}^l\gamma^F(\grad^F\sigma^k)s_k,  \\
&\nabla^M_{\partial_u}\sigma=\sum_{k=1}^l(\partial_u\sigma_k)s_k,  \\
&\Delta^F_s\sigma=\sum_{k=1}^l(\Delta_0\sigma^k)s_k,  \\
&(\nabla^M_{\partial_u})^2\sigma=\sum_{k=1}^l(\partial^2_u\sigma^k)s_k,  \\
&[D^F,\nabla^M_{\partial_u}]\sigma=\sum_{k=1}^l\gamma^F([\grad^F,\nabla^M_{\partial_u}]\sigma^k)s_k,
\end{split}
\end{equation}
\noindent and for any function $\varphi$
\begin{equation}
\begin{split}
[\grad^F,\nabla^F_{\partial_u}]\varphi=&-(f\partial_u(f^2)+\partial_uf)\partial_t\varphi f^{-1}\partial_t+\\
&-(h\partial_u(h^2)+\partial_uh)\partial_\theta\varphi h^{-1}\partial_\theta.
\end{split}
\end{equation}
\noindent Therefore,
\begin{equation}
\begin{split}
&[D^F,\nabla^M_{\partial_u}]\sigma=-\sum_{i,k=1}^l\{(f\partial_u(f^2)+\partial_uf)\underbrace{\left<\gamma^F(f^{-1}\partial_t)s_i,s_k\right>}_{=c_{ik}}\partial_t\sigma^i+  \\
&\qquad\qquad\qquad+(h\partial_u(h^2)+\partial_uh)\underbrace{\left<\gamma^F(h^{-1}\partial_\theta)s_i,s_k\right>}_{=d_{ik}}\partial_\theta\sigma^i\}s_k.
\end{split}
\end{equation}
Remark that $c_{rr}=d_{rr}\equiv 0$. In fact, $\nabla^Fs_r=0$ and thus $D^Fs_r=0$. Therefore for any $u_0\in[r_0,R_0]$ and any open $G\subset U_{u_0}\subset F_{u_0}$:
\begin{equation}(\underbrace{D^{F_{u_0}}s_r}_{=0},s_r)=(s_r,\underbrace{D^{F_{u_0}}s_r}_{=0})-\int_{\partial G}\left<\gamma^F(\nu)s_r,s_r\right>\dvol_{\partial G}.\end{equation}
If we denote by $\alpha:=g(\nu,f^{-1}\partial_t)$ and by $\beta:=g(\nu,h^{-1},\partial_{\theta})$, we obtain for all $G$:
\begin{equation}\int_{\partial G}(\alpha c_{rr}+\beta d_{rr})\dvol_{\partial G}=0\end{equation} Therefore, $c_{rr}=d_{rr}=0$.
Thus, the statement (i) follows. Statement (ii) follows by direct insertion and the properties of the frame $\{s_1,\dots,s_l\}$.\\
\end{proof}
Since we are primarly interested in the first few eigenvalues, and the equations and the boundary conditions are linear, we can choose for any $i\neq k$ $\;\sigma^i:=0$. Therefore we obtain for the lower absolute eigenvalues:
\begin{equation}
{\small
\begin{cases}(\Delta_0-\lambda)\sigma^k-\partial_u^2\sigma^k+(\partial_uH-H^2)\sigma^k+2H\partial_u\sigma^k=0.\\
(\partial_u-H)\sigma^k=0\quad(u=r_0,R_0)\end{cases}(\frac{l}{2}+1\le k \le l)
}
\end{equation}
or
\begin{equation}
{\small
\begin{cases}(\Delta_0-\lambda)\sigma^k-\partial_u^2\sigma^k+(\partial_uH-H^2)\sigma^k+2H\partial_u\sigma^k=0\\
\sigma^k=0\quad(u=r_0,R_0)\end{cases}(1\le k \le \frac{l}{2}).
}
\end{equation}
We are going now to explicitly determine a regular discrete resolution for the function Laplacian $\Delta_0$ on $\mathbf{R}^2$ with metric given in polar coordinates by\\ $f^2(u)dt^2+h^2(u)d\theta^2$ under the periodicity conditions $a(t,\theta)=a(t+\varepsilon,\theta+\rho)$ and $a(t,\theta)=a(t,\theta+2\pi)$ for all $(t,\theta)$.
\begin{lemma}\label{lemma56}
A regular spectral decomposition of $\Delta_0$ on $F_u$ is given by  $(g_i,\kappa_i)_{i\in\mathbf{Z}^2}$ i.e. $\Delta_0 g_i=\kappa_i g_i$ for all $i\in \mathbf{Z}^2$ and $(g_i)_{i\in\mathbf{Z}^2}$ is an o.n.b. of $L^2(F_u,\mathbf{C})$, where for $i=(r,s)\in\mathbf{Z}^2$:
\begin{equation}
\begin{split}
g_i&=g_i(u,t,\theta)=\frac{e^{\imath[(2\pi s+r\rho)\frac{t}{\varepsilon}-r\theta]}}{\sqrt{2\pi\varepsilon fh(u)}}\\
\kappa_i&=\kappa_i(u)=\frac{(2\pi s+r\rho)^2}{f^2(u)\varepsilon^2}+\frac{r^2}{h^2(u)}.
\end{split}
\end{equation}
\end{lemma}
\begin{proof}
 We have to solve the partial differential equation
\begin{equation}\Delta_0a=\kappa a\end{equation} for an unknown function of two variables $a=a(t,\theta)$, satisfying the periodicity conditions
\begin{equation}a(t,\theta)=a(t+\varepsilon,\theta+\rho)\quad\text{and}\quad a(t,\theta)=a(t,\theta+2\pi)\quad\text{for all $(t,\theta)$.}\end{equation}
The Ansatz $a(t,\theta)=A(t)B(\theta)$ inserted in $\Delta_0a=\kappa a$ leads to two ordinary differential equations
\begin{align}
&A^{\prime\prime}+f^2\mu A=0 \\
&B^{\prime\prime}+h^2\nu B=0,
\end{align}
and $a:=AB$ is then a solution of the originary PDE with $\kappa=\mu+\nu$. In view of the periodicity conditions, we ignore the cases where $\mu<0$ or $\nu<0)$ and find that $A(t)=e^{\imath f\sqrt{\mu}t}$ and $B(\theta)=e^{\imath h\sqrt{\nu}\theta}$ are solutions. We insert $a(t,\theta):=A(t)B(\theta)$ in the second periodicity condition to obtain
\begin{equation}\nu=\frac{r^2}{h^2}\quad\text{for $r\in\mathbf{Z}$}\end{equation}
and in the first
\begin{equation}\mu=\frac{(2\pi s+r\rho)^2}{\varepsilon^2 f^2}\quad\text{for a $s\in\mathbf{Z}$.}\end{equation}
The eigenvalues are therefore, setting $i:=(r,s)\in\mathbf{Z}^2$,
\begin{equation}\kappa_i=\kappa_i(u)=\mu+\nu=\frac{(2\pi s+r\rho)^2}{f^2(u)\varepsilon^2}+\frac{r^2}{h^2(u)}\end{equation} and the eigenfunctions
\begin{equation}a_i=a_i(t,\theta)=e^{\imath[(2\pi s+r\rho)\frac{t}{\varepsilon}-r\theta]}.\end{equation}
To get an o.n. sequence we normalize as follows:
\begin{equation}g_i(u,t,\theta):=\frac{a_i}{\|a_i\|}.\end{equation}
Since\begin{equation}\|a_i\|^2_{L^2(F_u,\mathbf{C})}=2\pi fh(u),\end{equation} we find $g_i$ as claimed.  The sequence $(g_i)_{i\in\mathbf{Z}^2}$ is an o.n.b. in $L^2(F_u,\mathbf{C})$.\\
\end{proof}
\noindent By direct verification we obtain the following
\begin{lemma}\label{lemma57}
The eigenfunctions $(g_i)_{i\in\mathbf{Z}^2}$ of $\Delta_0$ have the following properties under derivation:
\begin{equation}
{\small
\begin{array}{ll}
\partial_ug_i=-\frac{1}{2}\partial_u(\log(fh))g_i &\partial_u^2g_i=[-\frac{1}{2}\partial_u^2(\log(fh))+\frac{1}{4}(\partial_u(\log(fh))^2]g_i \\
  & \\
\partial_tg_i=\imath\frac{2\pi s+r\rho}{\varepsilon}g_i &\partial_t^2g_i=-\left(\frac{2\pi s+r\rho}{\varepsilon}\right)^2g_i  \\
 & \\
\partial_\theta g_i=-\imath rg_i &\partial_\theta^2g_i=-r^2g_i.
\end{array}
}
\end{equation}
\end{lemma}

\noindent Now we decompose the $k$-th coordinate function of the spinor $\sigma$ in its Fourier serie w.r.t. the o.n.b $(g_i)_{i\in\mathbf{Z}^2}$ of $L^2(F_u,\mathbf{C})$ found in Lemma \ref{lemma56}
\begin{equation}\sigma^k=\sum_{i\in\mathbf{Z}^2}a_i^kg_i\end{equation}where $a_i$ is the $u$-dependent $i$th Fourier coefficient.
We insert this decomposition in the boundary value problems found at the end of the preceding subsection and drop the $k$ superscript, because they all have the same form, independently of what $k\in\{1,\dots,l\}$ we consider. Using the properties of all the $g_i$s under derivation, listed in Lemma \ref{lemma57}, we obtain
\begin{equation} \begin{cases}\sum_{i\in\mathbf{Z}^2}[-a_i^{\prime\prime}+(\kappa_i-\lambda)a_i]g_i=0\\\sum_{i\in\mathbf{Z}^2}[a_i^{\prime}-(\log fh)^\prime a_i]g_i=0\quad(u=r_0,R_0)\end{cases}
\end{equation}
or
\begin{equation}
\begin{cases}\sum_{i\in\mathbf{Z}^2}[-a_i^{\prime\prime}+(\kappa_i-\lambda)a_i]g_i=0\\\sum_{i\in\mathbf{Z}^2}a_ig_i=0\qquad\qquad(u=r_0,R_0).\end{cases}
\end{equation}
All the equations are satisfied, if and only if  all the Fourier coefficients vanish. This leads to the following two families of $1$-dimensional  boundary value problems:
\begin{equation}\label{abs1}
\begin{cases}-a_i^{\prime\prime}+(\kappa_i-\lambda)a_i=0\\a_i^{\prime}-(\log fh)^\prime a_i=0\quad(u=r_0,R_0)\end{cases} \quad(i\in\mathbf{Z}^2).\\
\end{equation}
and
\begin{equation}\label{abs2}
\begin{cases}-a_i^{\prime\prime}+(\kappa_i-\lambda)a_i=0\\a_i=0\quad(u=r_0,R_0)\end{cases}\quad(i\in\mathbf{Z}^2).
\end{equation}

\begin{remark}\label{remk}
 All the eigenvalues of the original absolute eigenvalue equations for the piece of the tube are eigenvalues of these two families of $1$-dimensional boundary value problems, but not viceversa. In fact to get the eigenspinors for $\lambda$ on each $O$ (cf. Remark \ref{remark61}), it suffices to take the restrictions of an eigenspinor for $\lambda$ on all $U$ to $O$. The converse procedure does not work in general, because eigenspinors for the same eigenvalue $\lambda$ on different open subsets of $U$ do not necessarily need to match on the overlaps.\par
 The prominent example is $\lambda=O\left(\frac{1}{R_0-r_0}\right)$ corresponding to the choice $i=0\in\mathbf{Z}^2$ and $\kappa_0=0$, which is not an absolute eigenvalue of $D^M$ on $U$, because the restriction of the the spin structure on $M$ to the torus $F_u$ is non trivial, i.e., it does not admit harmonic spinors, as explained in \cite{Ba00}
\end{remark}

The boundary value problems for $i\neq0\in\mathbf{Z}^2$ give rise to eigenvalues, which are bounded away from $0$ uniformly w.r.t. $R$. We insert $f(u)=\cosh(R-u)$ and $h(u)=\sinh(R-u)$ and set for any $i\neq0\in\mathbf{Z}^2$
\begin{equation}
q_i(u):=\kappa_i(u)-\lambda=\frac{(2\pi s+r\rho)^2}{\cosh^2(R-u)\varepsilon^2}+\frac{r^2}{\sinh^2(R-u)}-\lambda\;\,\text{for $i=(r,s)$.}\end{equation}
Recall from Section \ref{curva} that $\varepsilon$, $\rho$, $R$ can't be arbitrarily chosen but have instead to satisfy the inequalities (\ref{ineqTube})
for positive constants $D_{1,2}$ and $E_{1,2}$. This fact implies a certain behaviour for the eigenvalues $\kappa_i(u)$ of the function Laplacian of the tube fibers $F_u$. There exists a positive constants $S$ such that for every $R\ge S$ and every $i\in\dot{\mathbf{Z}^2}$
\begin{equation}\kappa_i(u)\ge {\left(\frac{E_1}{D_2}\,e^{r_0}\right)}^2\qquad\forall u\in[r_0,R].\end{equation}
If we choose $r_0$ big enough, then for any $i\in\dot{\mathbf{Z}^2}$ $q_i(u)\ge 5-\lambda$ on $[r_0,R_0]$. Let us choose $R_0:=R-1$. If $\lambda<1$ is an absolute eigenvalue, solution of (\ref{abs1}), then there is a non trivial solution $a_i$ and $q_i(u)>4$ on $[r_0,R-1]$. By Proposition \ref{propA1} one has
\begin{equation}
\liminf_{R\rightarrow+\infty}\frac{a^\prime_i(R-1)}{a_i(R-1)}\ge1>0.
 \end{equation}
But this contradicts the absolute boundary condition at $u=R-1\rightarrow+\infty$ (as $R\rightarrow+\infty$), because
\begin{equation}\frac{a_i^\prime(R-1)}{a_i(R-1)}=-(\tanh(1)+\tanh^{-1}(1))<0.\end{equation}
If $\lambda<1$ is an absolute eigenvalue, solution of (\ref{abs2}), then there is a non trivial solution $a_i$ and $q_i(u)>4$ on $[r_0,R-1]$. By Proposition \ref{propA2} one has
\begin{equation}
a_i(R-1)\neq 0.
 \end{equation}
But this contradicts the absolute boundary condition at $u=R-1$, because
\begin{equation}a_i(R-1)=0.\end{equation}
The conclusion is that there are two positive constants $S$ and $r_0$ such that for all $R\ge S$, any absolute eigenvalue  must be greater than or equal to $1$. In terms of the sequence of pieces of tubes converging to a cusp this means
\begin{proposition}\label{prop58}
There exist an integer $j_0\in\mathbf{N}_0$ and a positive constant $r_0$ such that $\forall j\ge j_0,\;\forall n\ge 0$
\begin{equation}\lambda_1((\Delta_{s_j}^{(U_j,g_j)})_{B_+})\ge 1,\end{equation}
where $U_j:=T([r_0,R_j-1])$ is the relevant piece of tube.
\end{proposition}

\subsection{Proof of the Lower Bound Inequality}\label{sketch}
We first sketch  the structure of the proof of Theorems
\ref{spinors3Manifold}. We can assume without loss of generality
that $M$ has only one cusp. We apply Theorem \ref{Cheeger} to the extrinsic Dirac operator as in Proposition \ref{propT} noting that
the spectral bound holds true for the intrinsic Dirac operator as well, because the spectra of both extrinsic and intrinsic Dirac Laplacians, possibly under the absolute boundary condition, are the same.
\begin{enumerate}
\item For every $j\ge 0$, we cover its approximating manifold $M_j$
with three $0$-codimensional submanifolds with boundary:
\begin{enumerate}
\item $W_j\supset(M_j)_{]\mu,\infty[}\cup\left\{x\in(M_j)_{]0,\mu]}\,|\,R_j\ge\dist(x,{\gamma}_j)\ge
R_j-r_0\right\}$: a compact neighborhood of the thick part of
$M_j$.
\item $U_j\supset\left\{x\in (M_j)_{]0,\mu]}\,|\,R_j-r_0\ge\dist(x,{\gamma}_j)\ge
1\right\}$: a relevant piece of the tube (a solid annular torus).
\item $V_j\supset\left\{x\in
(M_j)_{]0,\mu]}\,|\,1\ge\dist(x,{\gamma}_j)\right\}$: a tubular
neighborhood of the closed geodesic (a solid torus).
\end{enumerate}
The submanifolds can be chosen as the closure of a $\varepsilon$ neighbourhood (for a fixed small $\varepsilon$ ) of the sets specified on the right hand side. The constant $r_0>0$ is chosen according to Proposition \ref{prop58}. 
\item We compute the spectral bound given by Theorem \ref{Cheeger}.
\item We control the spectra of the bounded parts $W_j$ and $V_j$  under the absolute boundary conditions
      using spectral perturbation theory.
\item Since the metric of the tube converge to the metric on
the cusp, the lower eigenvalues of $P$ on the piece of cusp for
the absolute  boundary conditions converge to the the lower
eigenvalues of $P$ on $U_j$ under the absolute boundary condition.
\end{enumerate}

\begin{proof}[Proof of Theorem \ref{spinors3Manifold}]
Following the steps above we apply Theorem \ref{Cheeger} for the cover of the manifold, for which we have $N_1=N_2=0$, $N=1$ to obtain

\begin{equation}
\begin{split}
\lambda_1^2(P) \ge C_1&\left\{\frac{1}{\mu(W_j)}+\frac{1}{\mu(U_j)}+\frac{1}{\mu(V_j)}+\right.\\
&+4\left(\frac{C_j}{\mu(W_j\cap U_j)}+1\right)\left(\frac{1}{\mu(W_j)}+\frac{1}{\mu(U_j)}\right)+\\
&\left.+4\left(\frac{C_j}{\mu(U_j\cap V_j)}+1\right)\left(\frac{1}{\mu(U_j)}+\frac{1}{\mu(V_j)}\right)\right\}^{-1},
\end{split}
\end{equation}
\noindent for a $C_1>0$ and constants $C_j>0$ depending on the $C^1$ norm of a partition of unity subordinate to the chosen cover (cf. Theorem \ref{Cheeger}). The constants $C_j$ are bounded from above by a constant $C_2>0$. Now, we examine the different eigenvalues involved:
\begin{itemize}
\item The eigenvalues $\mu(W_j)$ and $\mu(W_j\cap U_j)$ are bounded from below by a positive constant independent of $j$ because $W_j$ converges to a closed $\varepsilon$ neighbourhood of the thick part $M_{\text{thick}}$, which is compact.
\item By Proposition \ref{prop58} there is a $j_0\in\mathbf{N}_0$ such that the eigenvalue $\mu(U_j)\ge 1$ for all $j\ge j_0$.
\item The eigenvalues $\mu(V_j)$ and $\mu(U_j\cap V_j)$ are uniformly bounded from below by a positive constant independent of $j$, because $V_j$ and $U_j\cap V_j$ are bounded.
\end{itemize}
\noindent We conclude that there exist a positive constant $c>0$ such that
\begin{equation}
\lambda_1^2(P) \ge c,
\end{equation}
and the proof is completed.
\end{proof}
\begin{remark}We can mimick this proof for the Laplace-Beltrami operator on $1$-forms and reobtain Jammes's result stated in Theorem \ref{thm43}. The only essential difference is that the first non zero eigenvalue on the tube converges to zero as the inverse of the square of the diameter as the manifold degenerates.
\end{remark}
\appendix
\section{Some Results about Second Order Boundary Value Problems}
\begin{proposition}\label{propA1}
Let the function $a=a(u)$ be a non trivial solution of the linear
second order boundary value problem
\begin{equation}
\begin{cases}
-a^{\prime\prime}+qa=0 \\ a^{\prime}(m_0)+\alpha a(m_0)=0,
\end{cases}
\end{equation} where $q\in C^\infty([m_0,m_1])$ is a smooth function satisfying $q>k^2$ for constants $k, \alpha\in \mathbf{R}$
such that $k>0$ and $\alpha\le k$. Then, for the unique solution $v$
of the initial value problem
\begin{equation}
\begin{cases}-v^{\prime\prime}+k^2v=0\\
v(m_0)=a(m_0)\\
v^{\prime}(m_0)=a^{\prime}(m_0)
\end{cases}
\end{equation}
the following inequality holds on $[m_0,m_1]$:
\begin{equation}\label{comp}
\frac{a^{\prime}}{a}\ge\frac{v^{\prime}}{v}.
\end{equation}
In particular
\begin{equation}\label{nobvp}
\liminf_{m_1\rightarrow+\infty}\frac{a^\prime(m_1)}{a(m_1)}\ge\frac{1}{2}k.
\end{equation}
\end{proposition}

\begin{proof}
 Without loss of generality we can prove the inequality on $[m_0,m_1[$ and choose $m_0:=0$ and $m_1=+\infty$. We need to distinguish several cases:

\begin{description}

\item [{\sc case 0}:] $a(0)=0$ never occurs. In fact, both cases $a(0)=0$ and $a^{\prime}(0)=0$
are excluded by the assumption on the non triviality of $a$ and by
the existence and uniqueness theorem for the solutions of ordinary
differential equations.

\item[{\sc case 1}:] $a(0)>0$.\\
Since $a^{\prime}(0)=-\alpha a(0)>-ka(0)$, we obtain
$v(u)=a(0)\cosh(ku)+\frac{a^{\prime}(0)}{k}\sinh(ku)>0$ $\forall
u\in[0,+\infty[$. With $w:=a^{\prime}v-av^{\prime}$ it follows
$w^{\prime}=(q-k^2)av$, $w(0)=0$ and
$w^{\prime}(0)=(q(0)-k^2)v^2(0)>0$. So,
$\varepsilon_1:=\sup\left\{u\in]0,+\infty[\;|\;w^{\prime}>0\;\text{on}\;]0,u[\right\}$
must belong to $]0,+\infty]$. If $\varepsilon_1<+\infty$, then by
continuity $w^{\prime}(\varepsilon_1)=0$.\\ Analogously, since
$a(0)>0$,
$\varepsilon_2:=\sup\left\{u\in]0,+\infty[\;|\;a>0\;\text{on}\;]0,u[\right\}$
must be in $]0,+\infty]$. If $\varepsilon_2<+\infty$, then by
continuity $a(\varepsilon_2)=0$. Set
$\varepsilon:=\min\{\varepsilon_1,\varepsilon_2\}$. On
$[0,\varepsilon[$ one has $w\ge0$, i.e.
$\frac{a^{\prime}}{a}\ge\frac{v^{\prime}}{v}$, being $a$ and $v$
positive. Integrating both sides of this inequality , one gets $a\ge
v$ on $[0,\varepsilon[$. So, on this interval one has
$w^{\prime}(u)=(q-k^2)a(u)v(u)\ge(q(u)-k^2)v^2(u)=(q(u)-k^2)a^2(0)\cosh^2(ku)$
and $a(u)\ge v=a(0)\cosh(ku)$. Assume now that $\varepsilon<\infty$.
There are two possibilities: if $\varepsilon=\varepsilon_1$, then by
continuity
$w^{\prime}(\varepsilon_1)=(q(\varepsilon_1)-k^2)a^2(0)\cosh^2(k\varepsilon_1)>0$;
if $\varepsilon=\varepsilon_2$, again by continuity
$a(\varepsilon_2)\ge a(0)\cosh(k\varepsilon_2)>0$. In both cases
there is a contradiction, so it must be $\varepsilon=\infty$. We
therefore come to the conclusion that
$\frac{a^{\prime}}{a}\ge\frac{v^{\prime}}{v}$ on $[0,+\infty[$.

\item[{\sc case 2}:] $a(0)<0$.\\
We set $\bar{a}:=-a$ and $\bar{v}:=-v$. Case 1 leads to
$\frac{\bar{a}^{\prime}}{\bar{a}}\ge\frac{\bar{v}^{\prime}}{\bar{v}}$
on $[0,+\infty[$, which means
$\frac{a^{\prime}}{a}\ge\frac{v^{\prime}}{v}$ on the same interval.

\end{description}
\noindent By solving the initial value problem for $v$, we can
determine $v$ and $v^\prime$ explicitly:
\begin{equation}
\begin{split}
v(u)&=a(m_0)\cosh(k(u-m_0))+\frac{a^{\prime}(m_0)}{k}\sinh(k(u-m_0))  \\
 &  \\
v^{\prime}(u)&=ka(m_0)\sinh(k(u-m_0))+a^{\prime}(m_0)\cosh(k(u-m_0)).  \\
\end{split}
\end{equation}
Since $v(u)\neq0$ for $u\in[m_0,m_1]$ we can write:
\begin{equation}
\frac{v^{\prime}(u)}{v(u)}=\,k\,\frac{a(m_0)\tanh(k(u-m_0))+\frac{1}{k}a^{\prime}(m_0)}{a(m_0)+\frac{1}{k}a^{\prime}(m_0)\tanh(k(u-m_0))}.
\end{equation}
We insert the boundary condition $a^{\prime}(m_0)+\alpha a(m_0)=0$
and simplify by $a(m_0)\neq0$:
\begin{equation}
\frac{v^{\prime}(u)}{v(u)}=\,k\frac{\tanh(k(u-m_0))-\frac{\alpha}{k}}{1-\frac{\alpha}{k}\tanh(k(u-m_0))}.
\end{equation}
Since $k>\alpha$, we obtain
\begin{equation}
\lim_{u\rightarrow+\infty}\frac{v^{\prime}(u)}{v(u)}=k
\end{equation}
 and the inequality (\ref{nobvp}) follows from the estimate (\ref{comp}).\\
\end{proof}

\begin{proposition}\label{propA2}
Let the function $a=a(u)$ be a non trivial solution of the linear
second order boundary value problem
\begin{equation}
\begin{cases}
-a^{\prime\prime}+qa=0 \\ a(m_0)=0,
\end{cases}
\end{equation} where $q\in C^\infty([m_0,m_1])$ is a smooth function satisfying $q>k^2$ for a constant $k>0$.
Then, for the unique solution $v$ of the initial value
problem
\begin{equation}
\begin{cases}-v^{\prime\prime}+k^2v=0\\
v(m_0)=0\\
v^{\prime}(m_0)=a^{\prime}(m_0)
\end{cases}
\end{equation}
the following inequality holds on $]m_0,m_1]$:
\begin{equation}
\frac{a^{\prime}}{a}\ge\frac{v^{\prime}}{v}.
\end{equation}
There exist $\delta>m_0$ such that
\begin{equation}\label{nobvp0}
\begin{split}
a(u)&\ge a(\delta)e^{\frac{k(u-\delta)}{2}}>0\quad(a^{\prime}(m_0)>0)\\
a(u)&\le
a(\delta)e^{\frac{k(u-\delta)}{2}}<0\quad(a^{\prime}(m_0)<0).
\end{split}
\end{equation}
\end{proposition}
\begin{proof}
Without loss of generality we can prove the inequality on
$[m_0,m_1[$ and choose $m_0:=0$ and $m_1=+\infty$. We need to
distinguish several cases:

\begin{description}

\item [{\sc case 0}:] $a^{\prime}(0)=0$ never occurs. Cf. {\sc case 0} in the
proof of Proposition \ref{propA1}.

\item[{\sc case 1}:] $a^{\prime}(0)>0$.\\
There exist a $\delta>0$ small enough such that
$a^{\prime}(\delta)>0$ and $a(\delta)>0$. Note that $\alpha
:=-\frac{a^{\prime}(\delta)}{a(\delta)}< k$. We can continue by
applying Proposition \ref{propA1} and obtain the result stated.
\item[{\sc case 2}:] $a(0)<0$.\\
Analogously to {\sc case 2} in the proof of Proposition
\ref{propA1}.

\end{description}
\end{proof}


\begin{thebibliography}{9}
\bibitem[AJ11]{AJ11} B. AMMAN and P. JAMMES, {\it The Supremum of Conformally Covariant Eigenvalues in a Conformal Class},  in Variational Problems in Differential Geometry (ed. Bielawski, Houston et Speight), Vol. 394, London Mathematical Society Lecture Note Series, (1-23), 2011.
\bibitem[B\"{a}91]{Ba91} C. B\"AR, {\it Das Spektrum von Dirac-Operatoren}, Bonner
mathematische Zeitschriften, 1991.
\bibitem[B\"{a}92]{Ba92} C. B\"AR, {\it Lower Eigenvalue Estimates for Dirac Operators}, Math. Ann. 293
(39-46), 1992.
\bibitem [B\"{a}96]{Ba96} C. B\"AR, {\it Metrics with Harmonic Spinors}, Geometric and
Functional Analysis 6, (899-942), 1996.
\bibitem [B\"a00]{Ba00} C. B\"AR, {\it The Dirac Operator on Hyperbolic Manifolds of Finite Volume},
J. Diff. Geometry 54, (439-488), 2000.
\bibitem[BGM05]{BGM05} C. B\"AR, P. GAUDUCHON and A. MORIANU {\it Generalized Cylinders in Semi-Riemannian and Spin Geometry}, Math. Z. 249 n0. 3,
 (545-580), 2005.
 \bibitem[BP91]{BP91} R. BENEDETTI and C. PETRONIO, {\it Lectures on Hyperbolic Geometry},
Springer Verlag, 1991.
\bibitem[Be73]{Ber73} M. BERGER, {\it Sur les premières valeurs propres des variétés riemanniennes}, Compositio Math. 26, (129-149), 1973.
 \bibitem[BGV96]{BGV96} N. BERLINE, E. GETZLER and M. VERGNE. {\it Heat Kernels and Dirac Operators},
Corrected Second Printing, Grundlehren der mathematischen
Wissenschaften, Springer Verlag, 1996.
\bibitem[Bl83]{Ble83} D. BLEECKER, {\it The Spectrum of a Riemannian Manifold with a Unit Killing Vector Field}, Trans. Amer. Math. Soc. 275, (409-416), 1983.
\bibitem [BW93]{BW93} B. BOO\SS/BAVNBEK and K.P. WOJCIECHOWSKI, {\it Elliptic Boundary
Problems for Dirac Operators}, Birkh\"auser, 1993.
\bibitem[BG92]{BG92} J.-P. BOURGUIGNON and P. GAUDUCHON, {\it Spineurs, Op\'erateurs de Dirac et Variations de M\'etriques}, Commun. Math. Phys. 144, (581-599), 1992.
\bibitem[BdM71]{BdM71} L. BOUTET DE MONVEL {\it Boundary Value Problems for Pseudodifferential Operators}, Acta Math. 126,
(11-55), 1971.
\bibitem[Ch84]{Cha84} I. CHAVEL. {\it Eigenvalues in Riemaniann Geometry}, Academic Press, 1984.
\bibitem[CD93]{CD93} I. CHAVEL and J. DODZIUK, {\it The Spectrum of Degenerating Hyperbolic
Manifolds of Three Dimensions}, J. Diff. Geometry 39, (123-127),
1993.
\bibitem[CC89]{CC89} B. COLBOIS and G. COURTOIS, {\it Les p\'etites valeurs propres des
vari\'et\'es hyperboliques de dimension 3}, Pr\'epublications de l'
Institut Fourier, Grenoble, 130, 1989.
\bibitem[CC89bis]{CC89bis} B. COLBOIS and G. COURTOIS. {\it Les valeurs propres inf\'erieures
\'a $\frac{1}{4}$ des surfaces de Riemann de petit rayon
d'injectivit\'e}, Comment. Math. Helv. 64, (349-362), 1989.
\bibitem[CD94]{CoDo94} B. COLBOIS and J. DODZIUK, {\it Riemannian Metrics with large $\lambda_1$}, Proc. Amer. Math. Soc. 122, (905-906), 1994.
\bibitem[CH93]{CoHi93} R. COURANT, D. HILBERT. {\it Methoden der mathematischen Physik}, Springer, 1993.
\bibitem[Do82]{Do82} J. DODZIUK, {\it Eigenvalues of the Laplacian on forms}, Proc. Am.
Math. Soc. 85, (438-443), 1982.
\bibitem[DG95]{DG95} J. DODZIUK and J. MC GOWAN, {\it The Spectrum of the Hodge Laplacian
for a Degenerating Family of Hyperbolic Three Manifolds}, Trans.
Amer. Math. Soc. 347, no. 6, 1985-1995, 1995.
\bibitem[Do80]{Do80} H. DONNELY, {\it On the Essential Spectrum of a Complete Riemannian
Manifold}, Topology 20, (1-14), 1981.
\bibitem[Fa98]{Fa98} S. FARINELLI {\it Spectra of Dirac Operators on a Family of Degenerating Hyperbolic Three Manifolds}, Diss. Math. Wiss. ETH Zürich,
 Nr. 12690, Ref.: E. Zehnder ; Korref. und Supervisor: B. Colbois, 1998.
\bibitem[FS98]{FS98} S. FARINELLI and G. SCHWARZ. {\it On the Spectrum of the Dirac Operator under Bounday Conditions}, J. Geom. Phys, (67-84), 1998.
\bibitem[GP95]{GP95} G. GENTILE and V. PAGLIARA. {\it Riemannian Metrics with Large First Eigenvalue on Forms of Degree p},
Proceedings of the American Mathematical Society, Vol. 123, No. 12,
(3855-3858), 1995.
\bibitem [Gil93]{Gi93} P. B. GILKEY, {\it On the Index of Geometric Operators for Riemannian
Manifolds with Boundary}, Adv. in Math. 102, (129-183), 1993.
\bibitem [Gi84]{Gi84} P. B. GILKEY, {\it Invariance Theory, the Heat Equation and the
Atiyah-Singer Index Theorem}, Publish or Perish,
1984.
\bibitem [Gil95]{Gi95} P. B. GILKEY, {\it Invariance Theory, the Heat Equation and the
Atiyah-Singer Index Theorem}, Second Edition, Studies in Advanced Mathematics, CRC Press,
1995.
\bibitem [Gin09]{Gin09} N. GINOUX, {\it The Dirac Spectrum}, Lecture Notes in Mathematics, Springer, 2009.
\bibitem[Go93]{Go93} J. MC GOWAN, {\it The p-Spectrum of the Laplacian on Compact
Hyperbolic Three Manifolds}, Math. Ann. 279, (725-745), 1993.
\bibitem[Gra97]{Gra97} O. GRANDJEAN. {\it Non-Commutative Differential Geometry}, ETHZ Dissertation (Diss. ETH No.
12292), 1997.
\bibitem[Gre70]{Gr70} P. GREINER, {\it An Asymptotic Expansion for the Heat Equation}, Global Analysis, Berkeley 1968, Proc. Symp. Pure Math. 16, (133-137),  Amer. Math. Soc., Providence, 1970.
\bibitem[Gre71]{Gr71} P. GREINER, {\it An Asymptotic Expansion for the Heat Equation}, Arch. Rat. Mech. Anal. 41, (163-218), 1971.
\bibitem[Gro79]{Gro79} M. GROMOV. {\it Hyperbolic manifolds according to Thurston and
J{\o}rgensen}, Sem. Bourbaki 546, (1-14), 1979.
\bibitem[Gru96]{Gr96} G. GRUBB, {\it Functional Calculus of Pseudodifferential Boundary Problems}, Second Edition, Birk\"auser, 1996.
\bibitem[He90]{Hej90} D. HEJHAL, {\it Regular b-Groups, Degenerating Riemann Surfaces and Spectral Theory},
Mem. Amer. Math. Soc. 437, 1990.
\bibitem[He70]{Her70} J. HERSCH, {\it Quatre propri\'{e}t\'{e}es isop\'{e}rim\'{e}triques des membranes sph\'{e}riques homogènes}, C. R. Acad. Sei. Paris Sér. A 270, (139-144), 1970.
\bibitem[HMR15]{HMR15} O. HIJAZI, S. MONTIEL and S. RAULOT, {\it A Holographic Principle for the Existence of Imaginary Killing Spinors}, Journal of Geometry and Physics
Vol. 91, (12-28), 2015.
\bibitem[H\"o85]{Ho85} L. H\"ORMANDER, {\it The Analysis of Linear Partial Differential Operators III}, Grundlehren der mathematischen Wissenschaften, Springer Verlag,
1985.
\bibitem[Ja12]{Ja12} P. JAMMES,  {\it Minoration du spectre des variétés hyperboliques de dimension 3}, Bull. Soc. math. France. 140 (2), (237–255), 2012.
\bibitem[Ji93]{Ji93} L. JI, {\it Spectral Degeneration of Hyperbolic Riemann Surfaces}, J.
Differential Geom. 38, (263-313), 1993.
\bibitem[JZ93]{JZ93} L. JI and M. ZWORSKI, {\it The Remainder Estimate in Spectral
Accumulation for Degenerating Hyperbolic Surfaces}, J. Funct. Anal.
114, (412-420), 1993.
\bibitem[LM89]{LM89} H. B. LAWSON and M.-L. MICHELSOHN, {\it Spin Geometry}, Princeton
University Press, 1989.
\bibitem[MP90]{MP90} R. MAZZEO and R. PHILLIPS, {\it Hodge Theory on Hyperbolic Manifolds},
Duke Math. J. 60, (509-559), 1990.
\bibitem[Mo56]{Mo56} C. B. MORREY, {\it A Variational Method in the Theory of Harmonic Integrals, II}, Amer.
J. Math., 78, (137–170), 1956.
\bibitem[Pf00]{Pf00} F. PF\"AFFLE, {\it The Dirac spectrum of Bieberbach manifolds}, J. Geom. Phys. 35
(367-385), 2000.
\bibitem[Sc95]{Sc95} G. SCHWARZ, {\it Hodge Decomposition — A Method for Solving Boundary Value Problems},
Lecture Notes in Math., 1607, Springer-Verlag, Berlin, 1995.
\bibitem[Sl66]{Sl66} R. T. SEELEY, {\it Singular integrals and Boundary Problems}, Amer. J. Math. 88, 1966, (781-809).
\bibitem[Sl69]{Sl69} R. T. SEELEY, {\it The Resolvent of an Elliptic Boundary Value Problem}, Amer. J. Math. 91, 1969, (889-920).
\bibitem[Ta83]{Ta83} S. TANNO, {\it Geometric Expressions of Eigen 1-Forms of the Laplacian on Spheres}, Spectral Riemannian Manifolds, Kaigai, Kyoto, (115-128), 1983.
\bibitem[Th79]{Th79} W. THURSTON, {\it The Geometry and Topology of $3$ Manifolds}, Princeton
Lecture Notes, 1979.
\bibitem[Ur79]{Ura79} H. URAKAWA, {\it On the Least Positive Eigenvalue of the Laplacian for Compact Group Manifolds}, J. Math. Soc. Japan 31, (209-226), 1979.
\bibitem[Wo87]{Wo87} S. WOLPERT, {\it Asymptotics of the Spectrum and the Selberg Zeta Function on the Space of Riemannian Surfaces},
Comm. Math. Phys. 112, (283-315), 1987.
\bibitem[Xu92]{Xu92} Y. XU, {\it Diverging Eigenvalues and Collapsing Riemannian Metrics}, Institute for Advanced Study, October 1992.
\bibitem[YY80]{YaYa80} P. YANG and S.-T. YAU, {\it Eigenvalues of the Laplacian of Compact Riemann Surfaces and Minimal Submanifolds}, Ann. Scuola Norm. Sup. Pisa Cl. Sei. (4) 7, (55-63), 1980.
\end{thebibliography}
\end{document}